\documentclass[reqno,11pt]{amsart}
\usepackage{ytableau}
\usepackage{amsthm,amsfonts, amssymb, amscd}
\usepackage[all,cmtip]{xy}
\usepackage{euscript}

\usepackage{tikz}
\usetikzlibrary{matrix,arrows}
\usetikzlibrary{positioning}
\usetikzlibrary{cd}

\usepackage[normalem]{ulem}

\usepackage{hyperref}
\hypersetup{%
    linktoc=page
}

\let\oldtocsection=\tocsection
\let\oldtocsubsection=\tocsubsection
\renewcommand{\tocsection}[2]{\hspace{0em}\oldtocsection{#1}{#2}}
\renewcommand{\tocsubsection}[2]{\hspace{1em}\oldtocsubsection{#1}{#2}}

\newtheorem*{THM}{Theorem}

\newtheorem{Thm}{Theorem}[section]
\newtheorem{Lem}[Thm]{Lemma}
\newtheorem{Cor}[Thm]{Corollary}
\newtheorem{Prop}[Thm]{Proposition}

\theoremstyle{remark}
\newtheorem{Rem}[Thm]{Remark}
\theoremstyle{remark}
\newtheorem{example}[Thm]{Example}
\theoremstyle{definition}

\theoremstyle{definition}

\theoremstyle{definition}

\newtheorem{Def}[Thm]{Definition}

\numberwithin{equation}{section}

\newcommand{\A}{\mathbb{ A}}

\newcommand{\R}{\mathbb{ R}}           
\newcommand{\C}{\mathbb{C}}           
\newcommand{\N}{\mathbb{ N}}           
\newcommand{\Z}{\mathbb{ Z}}           
\newcommand{\PP}{\mathbb{ P}} 
   
\newcommand{\V}{\mathbb{ V}}   
\newcommand{\QQ}{\mathbb{ Q}}        
\newcommand{\Ad}{\operatorname{Ad }}             
\newcommand{\ad}{\operatorname{ad}}             

\newcommand{\spec}{\operatorname{Spec}}

\renewcommand{\ker}{\operatorname{ker }}
\newcommand{\diag}{\operatorname{diag}}

\newcommand{\fb}{{\mathfrak b}}

\newcommand{\fg}{{\mathfrak g}}

\newcommand{\ft}{{\mathfrak t}}
\newcommand{\fu}{{\mathfrak u}}

\newcommand{\ga}{\alpha}

\newcommand{\gl}{\lambda}

 \newcommand{\cb}{\mathcal{B}}
\newcommand{\cc}{\mathcal{C}}

\newcommand{\ce}{\mathcal{E}}

\newcommand{\ch}{\mathcal{H}}
 \newcommand{\ci}{\mathcal{I}}

 \newcommand{\cl}{\mathcal{L}}
 \newcommand{\cm}{\mathcal{M}}
 \newcommand{\cn}{\mathcal{N}}
 \newcommand{\co}{\mathcal{O}}
 \newcommand{\cp}{\mathcal{P}}
 \newcommand{\cq}{\mathcal{Q}}

 \newcommand{\cv}{\mathcal{V}}
 \newcommand{\cw}{\mathcal{W}}
 
 \newcommand{\cy}{\mathcal{Y}}

\newcommand{\toric}{\mathcal{V}}

\newcommand{\map}{\eta}

\DeclareMathOperator{\inv}{inv}
\newcommand{\RST}{{\tt RST}}
\newcommand{\IRST}{{\tt IRST}}

\newcommand{\std}{{\tt std}}
\newcommand{\pairs}{\operatorname{pairs}}

\renewcommand{\mod}{\;\mathrm{mod}\,}

\newcommand{\vb}{\mathbf{b}}

\newcommand{\shq}{\underline{\C}}

\renewcommand{\tilde}{\widetilde}

\renewcommand{\bar}[1]{\overline{#1}}

\renewcommand{\hat}{\widehat}

\def\va{\mathbf{a}}

\def\vb{\mathbf{b}}

\def\v0{\mathbf{0}}

\newcommand{\mn}{\mathsf{n}}
\newcommand{\ms}{\mathsf{s}}
\newcommand{\mt}{\mathsf{t}}
\newcommand{\mx}{\mathsf{x}}

\newcommand{\spr}[1]{\widetilde{\cn}_{#1}}
\newcommand{\espr}[1]{\widetilde{\cm}_{#1}}

\newcounter{myenumi}
\renewcommand{\themyenumi}{$(\arabic{myenumi})$}

\title[Extended Springer Fibers]{Geometric and Combinatorial Properties of Extended Springer Fibers}

\author{William Graham}
\address{Department of Mathematics\\ University of Georgia\\ Boyd Research and Education Center\\ Athens, GA\\ 30602\\ USA
}
\email{wag@math.uga.edu}

\author{Martha Precup}
\address{Department of Mathematics\\ Washington University in St. Louis \\ One Brookings Drive\\ St. Louis, Missouri\\ 63130\\ USA  }
\email{martha.precup@wustl.edu}

\author{Amber Russell}
\address{Department of Mathematics, Statistics, and Actuarial Science,
Butler University, 
4600 Sunset Avenue, Indianapolis, Indiana 46208\\ USA}
\email{acrusse3@butler.edu}

\date{\today}

\begin{document}
\parskip=4pt \baselineskip=14pt

\begin{abstract} We consider a generalization of the Springer resolution studied in earlier work of the authors, called the extended Springer resolution.  In type $A$, this map plays a role in Lusztig's generalized Springer correspondence comparable to that of the Springer resolution in the Springer correspondence.  The fibers of the Springer resolution play a key part in the latter story, and connect the combinatorics of tableaux to geometry.  Our main results prove the same is true for fibers of the extended Springer resolution-- their geometry is governed by the combinatorics of tableaux.  In particular, we prove that these fibers are paved by affines, up to the action of a finite group, and give combinatorial formulas for their Betti numbers.  This yields, among other things, a simple formula for dimensions of stalks of the Lusztig sheaves arising in the study of the generalized Springer correspondence, and shows that there is a close
resemblance between each Lusztig sheaf and the Springer sheaf for a smaller group.
\end{abstract}

\maketitle

\tableofcontents

\section{Introduction}\label{sec.intro}
The Springer resolution $\mu: \tilde{\cn} \to \cn$ is a desingularization of the nilpotent cone $\cn$ in the  Lie algebra $\fg$ of a semisimple algebraic group $G$.  This resolution plays a prominent role in geometric representation theory, arising
from Springer's striking discovery that there is a deep connection, called the Springer correspondence, between Weyl group
representation and $G$-orbits on the nilpotent cone (see \cite{Springer1976}).  In particular, there is a graded Weyl group representation
on the cohomology of the fibers of this resolution, which are called Springer fibers.
In the type $A$ case---that is, when $\fg=\mathfrak{sl}_n(\C)$ is the Lie algebra of the special linear group $SL_n(\C)$,  with corresponding Weyl group $S_n$---the Springer fibers are parametrized by partitions of $n$. The number of irreducible components of the Springer fiber corresponding a partition $\lambda$ is equal to the number of standard tableaux of shape $\lambda$, and its Poincar\'e polynomial is the generating function for a particular inversion statistic on the set of all row-strict tableaux of shape 
$\lambda$; more details appear in Section~\ref{sec.paving} below.  In this way, Springer fibers embody the interaction between the combinatorics used to study $S_n$-representations, and the geometry of the Springer resolution.

This paper studies extended Springer fibers, which are fibers of the extended Springer resolution 
\[
\psi: \tilde{\cm} \to \tilde{\cn} \to \cn
\]
introduced by the first author in~\cite{Graham2019}.  In~\cite{GPR} the authors show that in type $A$, this map plays a role in Lusztig's generalized Springer correspondence comparable to that of the Springer resolution in the Springer correspondence.  In other words, the paper~\cite{GPR} opens the door to studying interactions between the combinatorics of tableaux and the geometry of the generalized Springer correspondence, via the extended Springer fibers. Our work below initiates this area of study.   

To state our main results we briefly introduce some notation; references and more information can be found in Section~\ref{sec.notation}.   The nilpotent cone in $\mathfrak{sl}_n(\C)$ has a stratification by nilpotent orbits, each equal to a conjugacy class in $\mathfrak{sl}_n(\C)$ of nilpotent matrices of Jordan type $\lambda$, where $\lambda$ is a partition of $n$. Given $\mx$ in the orbit $\co_\lambda$, let $\spr{\mx}:= \mu^{-1}(\mx)$ denote the corresponding Springer fiber.  Spaltenstein proved that Springer fibers are paved by affines~\cite{Spaltenstein1977}, and the cells in the affine paving are in bijection with row-strict tableaux of shape $\lambda$.  Using a particular paving constructed by the second author and Ji in~\cite{Ji-Precup2022}, we lift each cell in the affine paving of the Springer fiber to the extended Springer fiber $\espr{\mx}:=\psi^{-1}(\mx)$, and analyze the result.

The lifts of the cells are obtained in the following way.
The variety $\tilde{\cm}$ is constructed from $\tilde{\cn}$ using the affine toric variety $\toric$ corresponding to the character group of a fixed maximal torus $T \subseteq SL_n(\C)$ and the cone of $\R_{\geq 0}$-linear combinations of simple roots; see Section~\ref{sec.extended} below.   The construction implies that $\tilde{\cm}$ carries an action of the center $Z$ of $G$.
Our investigation of pavings of extended Springer fibers begins with a study of particular subvarieties of $\toric$ and their images under the $Z$-action.  There is a natural quotient map $\espr{\mx} \to \spr{\mx} = \espr{\mx}/Z$.  Using our results about the action of
$Z$ on $\toric$, we prove that 
that the inverse image of each affine cell in the paving of the Springer fiber is a disjoint union of varieties, each of which
is isomorphic to affine space modulo a finite group.  These varieties form the cells in what we refer to as an orbifold paving.  
To keep track of each constituent of the union, we introduce a new numeric measure for row-strict tableaux, called divisors; see~Definition~\ref{def.divisor}.  The key idea is to break each tableau into equal sized blocks labeled by consecutive numbers---it is this combinatorial structure that dictates the $Z$-action on the inverse image of the corresponding affine cell to $\tilde{\cm}_\mx$.
We can now state our main result.

\begin{THM}[See Theorem~\ref{thm.extended-paving} below] Suppose $\mx\in \co_\lambda$.  The extended Springer fiber $\espr{\mx}$ has an orbifold paving with cells indexed by pairs $(\sigma, i)$ such that $\sigma$ is a row-strict tableau of shape $\lambda$ and $0\leq i \leq d_\sigma-1$ with $d_\sigma$ the maximal divisor of $\sigma$.  Furthermore, the action of $Z$ on $\tilde{\cm}_\mx$ cyclically permutes these cells.
\end{THM}

This paving allows us to compute topological invariants of the extended Springer fibers using 
tableaux, as in the case of ordinary Springer fibers.
We obtain a number of consequences: formulas for the Poincar\'e polynomial of each $Z$-isotypic component of the extended Springer fiber (Theorem~\ref{thm.poincare}), for the Poincar\'e polynomial of each extended Springer fiber in terms of Poincar\'e polynomials of smaller rank Springer fibers (Corollary~\ref{cor.poincare}), and for the dimensions of the stalks of the Lusztig sheaves, which play an important role in Lusztig's generalized Springer correspondence (Theorem~\ref{thm.lusztigsheaf}).  
As a consequence, we show that there is a remarkable resemblance between each Lusztig sheaf and the Springer sheaf for a smaller group (see Corollary \ref{cor:smallergroup}).  Although this resemblance can be seen in examples computed using the Lusztig-Shoji algorithm (\cite{Lu86}, \cite{Sho87}), we are not aware of any published reference.  

Several of our combinatorial formulas below rely on the fact that divisors of row-strict tableaux interact well with an inversion statistic defined by Tymoczko for computing the Betti numbers of Springer fibers~\cite{Tymoczko2006, Precup-Tymoczko}.  Indeed, Corollary~\ref{cor.shift} below gives an inductive formula for the number of inversions of a row-strict tableau of size $n$ with divisor $d$ in terms of the number of inversions of a row-strict tableau of size $n/d$.  Although this result is purely combinatorial, it is suggestive of a geometric phenomenon, and has the potential to lend new insight into the geometry of Springer fibers, which we plan
to explore in future work.  Also, the existence of a paving of the extended Springer fibers suggests potential connections with parity sheaves \cite{JMW14}, which have been instrumental in recent work on the modular Springer correspondence \cite{AHJR1}.

We now summarize the contents of the paper.  Section~\ref{sec.notation} covers background information and definitions, including the definitions of the extended Springer resolution and of orbifold pavings. In Section~\ref{sec.affine-subvarieties}, we study subvarieties of the affine toric variety $\cv$, each indexed by a partition of the set $[n-1]$, and their image under the $Z$-action.  In Section~\ref{sec.paving}, we apply the results of Section~\ref{sec.affine-subvarieties} in the context of extended Springer fibers and prove our main paving theorem, Theorem~\ref{thm.extended-paving}.  Finally, Section~\ref{sec.combinatorics} explores the combinatorial consequences of the paving.

\
\noindent\textbf{Acknowledgments.} The second author is partially supported by NSF grant DMS 1954001.


\section{Background and notation}\label{sec.notation}

Throughout this paper, $G$ denotes the simply connected, semisimple algebraic group $SL_n(\C)$ with Lie algebra $\fg=\mathfrak{sl}_n(\C)$ and center $Z\simeq \Z_n$.  We fix $B \subseteq G$ to be the Borel subgroup of upper triangular matrices.  Then $B = TU$, where $U$ is the subgroup of upper 
triangular matrices with diagonal entries equal to $1$, and $T$ is the subgroup of diagonal matrices.  
Denoting the Lie algebras of 
of $B$, $T$ and $U$ by $\fb$, $\ft$, and $\fu$, respectively, we have $\fb=\ft\oplus \fu$.  The Weyl group of $G$ is 
$W= N_G(T)/T\cong S_n$, the symmetric group on $n$ letters.  Given $g \in G$ and $\mx \in \fg$, 
we write $g \cdot \mx = \Ad(g) \mx$ for the adjoint action of $g$ on $\mx$; in terms of matrix multiplication,
$g \cdot \mx = g \mx g^{-1}$.  Let $\cn$ denote the nilpotent cone of~$\fg$.

Let $E_{i,j}$ denote the elementary matrix whose only nonzero entry is  $1$ in position $(i,j)$.  Then
$\fu=\C\{ E_{i,j} \mid 1\leq i<j \leq n \}$.  We define $\epsilon_i\in \ft^*$ by $\epsilon_i(\mt) = \mt_{ii}$,
where $\mt = \sum \mt_{ii} E_{i,i} \in \ft$.  Let $\Phi$ denote
the set of roots of $\ft$ on $\fg$, and $\Phi^+ = \{ \epsilon_i - \epsilon_j \mid i < j \}$ denote
the set of positive roots.  The corresponding set of simple roots is 
$\Delta= \{\alpha_i \mid i\in \{1, \ldots, n-1\}\}$, where $\alpha_i:=\epsilon_i-\epsilon_{i+1}$.  

The center $Z$ of $G$ is isomorphic to $\Z_n$; it consists of scalar multiples of the $n \times n$ identity matrix $I_n$.
Let $\widehat{Z}$ denote the character group of $Z$, and let
$\chi_i \in \widehat{Z}$ be the character of $Z$ defined by $\chi_i(\diag(a,a,\ldots, a)) = a^i$.


\subsection{The Springer resolution} \label{sec.Springer.def} The \textit{Springer resolution}  is the morphism $\mu: \tilde{\cn} \to \cn$, where
\[
\tilde{\cn} := \{(gB, \mx) \in G/B \times \cn \mid g^{-1}\cdot \mx \in \fu \},
\]
and the map $\mu$ is the projection onto the second factor. 
We will make use of the following alternate description of the Springer resolution.  Write $G\times^B \fu$ for the mixed space $(G \times \fu)/B$, where
$B$ acts on $G\times \fu$ on the right by $(g,\mx)b = (gb, b^{-1} \cdot \mx)$.  Let $[g,\mx] := (g,\mx) B \in G\times^B \fu$.
There is an isomorphism of varieties
\[
G\times^B \fu \to \tilde{\cn},\; \;[g,\mx] \mapsto (gB, g\cdot \mx),
\]
(cf.~\cite[pg.~66]{Jantzen}),
which we use to identify these varieties.  Under this identification, the morphism $\mu$ is 
\begin{eqnarray}\label{eqn.Sp.res}
\mu: G\times^B \fu \to \cn, \; \; [g,\mx] \mapsto g\cdot \mx.
\end{eqnarray}

The Springer fiber corresponding to a nilpotent matrix
$\mx\in \cn$ is the fiber $\mu^{-1}(\mx)$ of the Springer resolution $\mu: G \times^B \fu \to {\mathcal N}$,
$$
\spr{\mx}:=\mu^{-1}(\mx) = \{ [g, g^{-1} \cdot \mx ] \mid g^{-1} \cdot \mx \in \fu \}.
$$
The natural map $G \times^B \fu \to G/B$ maps $\spr{\mx}$ isomorphically onto its image in the flag variety, which
is the subvariety
$\cb_{\mx}$ of $\cb = G/B$ defined by 
$$
\cb_{\mx} = \{ gB \mid g^{-1} \cdot \mx \in \fu \}.
$$
We identify $\spr{\mx}$ with $\cb_{\mx}$, generally writing $\spr{\mx}$ throughout.  

Let $\lambda$ be a partition of $n$.  We denote by $\co_\lambda$ the nilpotent orbit of matrices of Jordan type $\lambda$. 
When $\mx$ is a fixed element of $\co_\lambda$ we write $\spr{\gl}:= \spr{\mx}$.  The Springer fiber
$\spr{\gl}$ is equidimensional of dimension $\sum_i (\gl_i - 1)$, where the $\gl_i$ are the parts of $\gl$.
A partition is divisible by $d$ if all of its parts are divisible by $d$; in this case, we define $\lambda/d$ to be
the partition whose parts are $\gl_i/d$.  Then $\spr{\gl/d}$ denotes the Springer fiber for the group $SL_{n/d}(\C)$
over a fixed nilpotent element of Jordan type $\gl/d$.


\subsection{The extended Springer resolution} \label{sec.extended}  We now recall from~\cite{Graham2019, GPR} the definition of the variety which is the main geometric focus of this paper.   More details may be found in these references.

The nilpotent matrix $\mn:=\sum_{i=1}^{n-1} E_{i,i+1}\in \fu$ is a regular nilpotent element of $\fg$.  
We can find a standard triple in $\fg$ with nilpositive element $\mn$ and semisimple element $\ms$ such that
$[\ms, E_{i,i+1}] = 2 E_{i,i+1}$.  Writing $\fg_2$ for the space spanned by the $E_{i,i+1}$,
we see that $[\ms, \mx] = 2\mx$ for all $\mx\in \fg_2$.

Let $\omega := \exp(2\pi i/n)$, and let $\underline{\omega} = \omega I_n$, where $I_n$ is the $n \times n$ identity matrix.
The center $Z$ of $G$ is the subgroup of $T$ given by $Z = \{\underline{\omega}^k  \mid 0\leq k \leq n-1\}$.  Let $T_{ad}$ denote the torus $T/Z$.
Since the center $Z$ acts trivially on $\fg$, the action of $T$ on $\fg$ factors through
the map $T \to T_{ad} = T/Z$.   The map $T_{ad} \to T_{ad} \cdot \mathsf{n}$ given by $t \mapsto t \cdot \mathsf{n}$ is
a $T$-equivariant open embedding of
$T_{ad}$ in $\fg_2$ as an open subset.  Hence,
 $\fg_2$ is an affine toric variety for $T_{ad}$; to emphasize this fact, we 
 define $\toric_{ad}:=\fg_2$.
The composition $\toric_{ad} \to \fu \to \fu/[\fu, \fu]$ is an isomorphism, and using this,
we identify $\toric_{ad}$ with $\fu/[\fu,\fu]$.  Via this identification, $\toric_{ad}$
acquires a $B$-action (the subgroup $U$ acts trivially), and the projection
$p: \fu \to \fu/[\fu,\fu]=\toric_{ad}$ is $B$-equivariant.

An affine toric variety is characterized by the character group of the torus,
which can be viewed as a subset of the dual Lie algebra of the torus, together with a cone in the real span of 
the character group.  Note that the simple roots can be viewed
as characters of either $T_{ad}$ or $T$.   Let $\cp\subset \ft^*$ denote the weight lattice, that is, the
character group of $T$, and let $\cq\subset \ft^*$ denote the root lattice, the character group
of $T_{ad}$.  Both have the same real span $\V$ in $\ft^*$, and each is a lattice in $\V$.  
For later use, we note that
the lattice $\frac{1}{n} \cq$ in $\V$ is the character group of a torus which we denote by~$\hat{T}$.  

The toric variety $\toric_{ad}$ corresponds
to the lattice $\cq$, and the cone equal to the
set of $\R_{\geq 0}$-linear combinations of simple roots.   We define $\toric$ to be the toric variety for $T$ 
obtained by using the lattice $\cp$ in place of $\cq$, but keeping the same cone.
It follows from \cite[Section 2.2]{Fulton}
that $\toric/Z = \toric_{ad}$ (see
\cite[Prop.~3.1]{Graham2019}).  We use the natural projection $B \to T \cong B/U$ to extend
the $T$-action on $\toric$ to a $B$-action, where
the subgroup $U$ acts trivially.  The quotient map $\pi:\toric
\to \toric_{ad}$ is $B$-equivariant.  For later use, let $\widehat{\toric}$ denote the toric
variety for $\hat{T}$ defined using the same cone, but the lattice~$\frac{1}{n} \cq$.

The variety $\tilde{\cm}$ and the map $\psi: \tilde{\cm} \to \cn$ discussed in the introduction are defined as follows.  First, consider the maps $p: \fu \to \fu/[\fu,\fu]=\toric_{ad}$ and $\pi: \toric \to \toric/Z=\toric_{ad}$ defined above.  Set
\[
\tilde{\fu} := \toric\times_{\toric_{ad}} \fu =\{ (v, \mathsf{x})\mid \pi(v)=p(\mathsf{x}) \};
\] 
that is, we form the following Cartesian diagram:
\[
\xymatrix{\tilde{\fu} \ar[r] \ar[d] & \fu \ar[d]^p  \\ \toric  \ar[r]^{\pi}& \toric_{ad}. }
\]
Because the maps $p$ and $\pi$ are both $B$-equivariant, $B$ acts on $\tilde{\fu}$.  We define $\tilde{\cm}:=G\times^B\tilde{\fu}$.  Let $\map: \tilde{\fu} \to \fu$ denote  projection onto the second factor.  Define $\tilde{\map}: \tilde{\cm} \to \tilde{\cn}$ to be the map induced from $\map$, so $\tilde{\map}$ maps the element $[g,y]\in \tilde{\cm}$ to $[g, \map(y)]\in \tilde{\cn}$.  The \textit{extended Springer resolution} is the variety $\tilde{\cm}$, together with the map $\psi: \tilde{\cm}\to \cn$, where $\psi$ is the composition
\[
\xymatrix{ \psi: \tilde{\cm} \ar[r]^{\tilde{\map}} & \tilde{\cn} \ar[r]^{\mu} & \cn} 
\]
of $\tilde{\map}$ with the usual Springer resolution $\mu$.  The \textit{extended Springer fiber}
corresponding to $\mx\in \cn$ is the fiber $\espr{\mx} := \psi^{-1}(x)$.
As for Springer fibers, if $\mx$ is a fixed matrix of Jordan type $\lambda$, we write $\espr{\gl}:=\espr{\mx}$.


\subsection{Pavings} \label{sec.praving.def} The main result of this paper shows that each extended Springer fiber has an orbifold paving by affine cells.  We now explain this terminology.

Recall that a \textit{paving} of an algebraic variety $\cy$ is a filtration by closed subvarieties
\[
\cy_0 \subset \cy_1 \subset \cdots \subset \cy_d=\cy.
\]
A paving is \textit{affine} if every $\cy_i\setminus \cy_{i-1}$ is isomorphic to a finite disjoint union of affine spaces.  We call these spaces the \textit{affine cells} of the paving.  We say a paving is an \textit{orbifold paving} if every $\cy_i\setminus \cy_{i-1}$ is isomorphic to a finite disjoint union of quotients of an affine space by a finite group action.

It is well-known that if an algebraic variety $\cy$ has a paving by affines, then the fundamental classes of the closures of the
cells in the paving form a basis for the Borel-Moore homology of $\cy$ over the integers (here denoted by $\overline{H}_*(\cy;\Z)$).  This result follows by induction, using
the long exact sequence in Borel-Moore homology, together with the fact that $\overline{H}_{2n}(\C^n) = \Z \cdot [\C^n]$,
and $\overline{H}_i(\C^n) = 0$ if $i \neq 2n$.   As shown by Abe and Matsumura \cite[Section 2.4]{Abe-Matsumura2015}, if rational
coefficients are used, this
result extends to orbifold
pavings: if $\cy$ has an orbifold paving, the fundamental classes of orbifold cell closures form a basis for the rational
Borel-Moore homology $\overline{H}_*(\cy;\QQ)$.
The reason is that if
$H$ is a finite group acting on $\C^n$, then $\overline{H}_{2n}(\C^n/H;\QQ) = \QQ \cdot [\C^n/H]$.  

With an additional hypothesis on the group action, these results can be extended to a more
general coefficient field.  (We do not know if this additional hypothesis is necessary.)
It is well known that if $\pi: \cy \to \cy/H$
is a quotient map by a finite group, then in ordinary homology, the map $\pi_*: H_i(\cy;\mathbb F) \to H_i(\cy/H;\mathbb F)$ is an isomorphism, if
coefficients are taken in a field $\mathbb F$ whose characteristic is $0$ or is relatively prime to $|H|$.  For lack of reference,
we prove the following lemma, which is an analogue of this result in Borel-Moore homology.

\begin{Lem}
Let $H$ be a finite abelian group acting linearly on $\C^n$.  If  $\mathbb F$ is a field whose characteristic is $0$ or is relatively
prime to $|H|$, then $\overline{H}_{2n}(\C^n/H;\mathbb F) = \mathbb F \cdot [\C^n/H]$,
and $\overline{H}_i(\C^n/H;\mathbb F) = 0$ if $i \neq 2n$.
\end{Lem}

\begin{proof}
For simplicity, we omit the coefficient field $\mathbb F$ from the notation and write simply
$\overline{H}_i(\cy)$ for $\overline{H}_i(\cy;\mathbb F)$.  If $H$ acts on $\cy$, 
write $\tilde{\cy} = \cy/H$. 

Let $\pi: \C^n \to \tilde{\C^n}$ be the quotient map.
It suffices to prove that for all $i$, the map $\pi_*: \overline{H}_{i}(\C^n) \to  \overline{H}_{i}(  \tilde{\C}^n) $
is an isomorphism.
We may assume $H$ acts diagonally on $\C^n$.  Embed $\C^n \hookrightarrow \C^{n+1}$ via $x \mapsto (x,0)$, and 
extend the $H$-action to  $\C^{n+1}$ by making $H$ act trivially on the last factor.  Then we have
$$
\C^n \hookrightarrow \PP^n \hookleftarrow \PP^{n-1},
$$
where the horizontal maps are an open and closed embedding, respectively.  These maps induce embeddings
$$
\tilde{\C}^n \hookrightarrow \tilde{\PP}^n \hookleftarrow \tilde{\PP}^{n-1}.
$$
We have the following map of exact sequences:
$$
\begin{CD}
 \overline{H}_{i}(\PP^{n-1}) @>>> \overline{H}_{i}(\PP^{n}) @>>> \overline{H}_{i}(\C^n) @>>> \overline{H}_{i-1}(\PP^{n-1}) @>>> \overline{H}_{i-1}(\PP^{n}) \\
 @VVV @VVV @VVV  @VVV  @VVV \\
  \overline{H}_{i}(  \tilde{\PP}^{n-1}) @>>> \overline{H}_{i}(  \tilde{\PP}^{n} ) @>>> \overline{H}_{i}(  \tilde{\C}^n) 
  @>>> \tilde{H}_{i-1}(  \tilde{\PP}^{n-1}) @>>> \overline{H}_{i-1}(  \tilde{\PP}^{n}).
\end{CD}
$$
The vertical maps are induced by the quotient map.  For compact varieties, Borel-Moore homology is isomorphic
to ordinary homology, on which the quotient map induces isomorphism.  Because $\PP^n$ and $\PP^{n-1}$ are compact,
this observation implies
that all maps except the middle map are an isomorphism.  The five-lemma then implies that the middle map is
an isomorphism, as desired.
\end{proof}


\section{Affine toric varieties and quotients}  \label{sec.affine-subvarieties}

The toric varieties $\toric$ and $\toric_{ad}$ defined in the previous section play a key role in the study of extended Springer fibers.  In this section, we define and study particular subvarieties of $\toric$ associated to a given set partition.  We prove that each such subvariety is isomorphic to a quotient of affine space by a finite group action.  The
center $Z$ acts on the set of such subvarieties, and we identify the orbit of each under the action of $Z$. Our work in
this section lays the groundwork needed for Section~\ref{sec.paving}, where we lift the affine cells paving a Springer fiber to the extended Springer fibers.

\subsection{The toric varieties $\cv$, $\cv_{ad}$, and $\widehat{\cv}$}
The inclusions $\cq \subset \cp \subset \frac{1}{n} \cq$ of lattices correspond to inclusions 
\begin{equation} \label{e.inclusions}
\C[T_{ad}] \subset \C[T] \subset \C[\hat{T}],
\end{equation}
of coordinate rings, which in turn yield
surjections
\begin{equation} \label{e.torimaps}
\begin{tikzcd} 
\widehat T \arrow[rr,bend left,"\widehat \rho"] \arrow[r] \arrow[r, "\rho"] & T \arrow[r,"\pi"] & T_{ad} .
\end{tikzcd}
\end{equation}

We now introduce notation for functions on these tori.
Given a character $\gl$ of either $T$ or $T_{ad}$, let 
$e^{\gl}$ denote the corresponding
function on the torus.
Recall that $\Delta = \{ \ga_1, \ldots, \ga_{n-1} \}$ denotes the set of simple roots.  Let $\{ \gl_1, \ldots, \gl_{n-1} \}$ denote
the corresponding set of fundamental weights.  Write
$$
x_i = e^{\alpha_i}, \ \ y_i = e^{\lambda_i}, \ \ z_i = e^{\alpha_i/n}.
$$
Then $\C[T_{ad}] =  \C[x_1^{\pm 1},\dots, x_{n-1}^{\pm 1}]$, $\C[T] = \C[y_1^{\pm 1},\dots, y_{n-1}^{\pm 1}]$,
and $\C[\hat{T}] = \C[z_1^{\pm 1},\dots, z_{n-1}^{\pm 1}]$.  

We have $x_i = z_i^n$.  Moreover, the equation
\begin{eqnarray} \label{e.glk}
\gl_k &= & \frac{1}{n} \Big( (n-k) \ga_1 + 2(n-k) \ga_2 + \cdots +
k(n-k) \ga_k  \\ 
\nonumber & & \quad\quad + k(n-k-1) \ga_{k+1} + k(n-k-2) \ga_{k+2} + \cdots + k \ga_{n-1} \Big).
\end{eqnarray}
implies that
\begin{equation} \label{e.yk}
y_k = z_1^{n-k}z_2^{2(n-k)}\dots z_k^{k(n-k)}z_{k+1}^{k(n-k-1)}z_{k+2}^{k(n-k-2)}\dots z_{n-1}^k.
\end{equation}
This gives a precise relationship between the coordinate functions on $T$ and $T_{ad}$ and those on~$\hat{T}$.  

The kernel of $\pi$ is equal to the center $Z$ of $G$.  Let $\widehat{H} := \ker \widehat{\rho} \supset H := \ker \rho$.
The equation $x_i = z_i^n$ implies the map $\widehat{\rho}$ can be described in coordinates as
$$
\widehat{\rho}(z_1,\dots,z_{n-1}) = (z_1^n, \ldots, z_{n-1}^n).
$$
Hence $\widehat{H} \cong \mathbb Z_n\times\cdots \times \mathbb Z_n$ (with $(n-1)$ copies of $\mathbb Z_n$).
Recall that $\omega$ denotes a primitive $n$-th root of unity.  Then any element $h \in \widehat{H}$ is of the form
$h = (\omega^{a_1},\dots,\omega^{a_{n-1}})$, where $a_i \in \Z$.

\begin{Prop} \label{prop.H}
The group $H$ is equal to
\begin{equation} \label{e.H}
\{ (\omega^{a_1},\dots,\omega^{a_{n-1}}) \in \widehat T \mid a_1 + 2a_2+\cdots +(n-1)a_{n-1} \equiv 0\mod n\}.
\end{equation}
\end{Prop}
\begin{proof}
From \eqref{e.yk}, we see that  $(\omega^{a_1},\dots,\omega^{a_{n-1}})\in H$ if and only if 
$\omega^{N_k} = 1$ for every $k$ such that $1\leq k\leq (n-1)$, where
\begin{eqnarray*}
N_k &=& a_1(n-k)+2a_2(n-k)+\cdots +ka_k(n-k)\\
&&\quad \quad \quad\quad +ka_{k+1}(n-k-1) +ka_{k+2}(n-k-2)+\cdots +ka_{n-1}.
\end{eqnarray*}
Equivalently, the tuple $(a_1, \ldots, a_{n-1})$ is such that the modular relation
$N_k \equiv 0 \mod n$ are satisfied for all $1\leq k\leq n-1$. This in turn is equivalent to the tuple satisfying the relations
$$
k a_1 + 2 k a_2 + \cdots + (n-1) k a_{n-1}  \equiv 0 \mod n
$$
for all $1\leq k\leq (n-1)$.  These relations are satisfied for all such $k$ if and only if they are satisfied for $k =1$,
which implies the result.
\end{proof}

We can identify $Z$ with $\widehat{H}/H$, and do so explicitly as follows.  First, as $Z = \{\underline{\omega}^k \mid 0\leq k \leq n-1\}$, we see $Z$ is
isomorphic to the cyclic group $\langle \omega \rangle \simeq \Z_n$.  By Proposition~\ref{prop.H} the homomorphism
\begin{equation} \label{e.hmap}
\widehat{H} \rightarrow Z,\;\; (\omega^{a_1}, \ldots, \omega^{a_{n-1}}) \mapsto \underline{\omega}^{\sum_{r=1}^{n-1} r a_r}
\end{equation}
has kernel $H$, and thus yields the desired isomorphism of $\widehat{H}/H$ with $Z$.

Recall from Section \ref{sec.extended} that the cone in $\V$ which is equal to
the $\R_{\geq 0}$-linear combinations of the positive simple roots defines an affine
toric variety for each of the tori $\widehat{T}$, $T$, and $T_{ad}$.  These toric varieties are denoted
by $\toric$, $\toric_{ad}$, and $\widehat{\toric}$, respectively.

As for the tori, the inclusions $\C[\toric_{ad}] \subset \C[\toric] \subset \C[\hat{\toric}]$
yield maps of toric varieties which extend the maps of the tori, and which
we denote by the same letters:
\begin{equation} \label{e.toricmaps}
\begin{tikzcd} 
\widehat \toric \arrow[rr,bend left,"\widehat \rho"] \arrow[r] \arrow[r, "\rho"] & \toric \arrow[r,"\pi"] & \toric_{ad}.
\end{tikzcd}
\end{equation}
The maps $\widehat{\rho}$, $\rho$, and $\pi$, are quotient maps for the actions of the finite
groups $\widehat{H}$, $H$, and $Z$, respectively.

To simplify notation, we write $A = \C[\widehat{\cv}]$. Then $\C[\cv] = A^H$, since $\cv \cong \hat\cv /H$.  
For each $i$ with $1\leq i \leq n-1$, let $\mu_i = \sum a_i \alpha_i$ denote the unique weight in the 
same coset of $\lambda_i$ modulo the root lattice
such that each $a_i$ satisfies $0 \leq a_i < 1$, and set $v_i:= e^{\mu_i} \in \C[\toric]$. 
The rings
$A=\C[\hat{\toric}]$ and $\C[\toric_{ad}]$ are the polynomial rings
$\C[z_1, \ldots, z_{n-1}]$ and $\C[x_1, \ldots, x_{n-1}]$, respectively.  Thus, both $\hat{\toric}$ and $\toric_{ad}$ are
isomorphic to $\C^{n-1}$, with coordinate functions $z_i$ and $x_i$, respectively.  However, the ring $\C[\toric]$
is more complicated: it is a quotient of the polynomial ring 
$\C[v_1, \ldots, v_{n-1}, x_1, \ldots, x_{n-1}]$, but it
is not a polynomial ring (except if $n =2$).


\subsection{Affine spaces mod finite groups} \label{ss.aff-mod-finite}
In this section, we define certain subvarieties of $\toric$, and show that each such subvariety is the quotient
of affine space by a finite group.  These subvarieties will play a role in our paving of the generalized Springer fibers.

Fix a decomposition of the set $[n-1]:=\{1, \ldots, n-1 \}$ into three disjoint (possibly empty) subsets $I$, $J$, $K$. 
Define $\cw_{ad}$ to be the subvariety of $\toric_{ad}$ defined 
by the equations $x_j = 1$, $x_k = 0$ for $j \in J$ and $k \in K$.  There is a natural isomorphism
\begin{eqnarray}\label{eqn.affine.iso}
i: \C^{|I|} \to \cw_{ad}.
\end{eqnarray}
If $J$ is nonempty, we fix a tuple of integers $(c_j)_{j\in J}$ such that $c_j\in \{0,\ldots, n-1\}$.
Define $\widehat{\cw}_{(c_j)}\simeq   \C^{|I|}$ to be the affine subvariety in 
$\hat\cv$ defined by the equations $z_j = \omega^{c_j}$, $z_k = 0$, for $j \in J$, $k \in K$, and
set $ \cw_{(c_j)} := \rho(\widehat{\cw}_{(c_j)})$.  
If the integers $c_j$ are fixed and there is no possible confusion, we write $\hat\cw := \hat\cw_{(c_j)}$ and
$\cw:= \rho(\hat\cw) \subset \toric$.  

Observe that  $\cw_{ad} = \hat\rho(\hat\cw) \subset \toric_{ad}$.   The varieties $\widehat{\cw}_{(c_j)}$ are all
distinct for different $(c_j)$, but this is not true for the $ \cw_{(c_j)}$; see Proposition \ref{p.components} for a precise statement.
However, the subgroup of $H$ preserving $ \cw_{(c_j)}$ depends only on the set $J$ and not on the particular values
$c_j$.  Indeed, if we define the subgroup $H_J$ of $H$ by the equation
$$
H_J = \{ (\omega^{a_1},\dots,\omega^{a_{n-1}}) \in \widehat T \mid a_j \equiv 0 \mod n\, \mbox { for } j \in J,\,  \sum_{r  \not\in J} r a_r \equiv 0 \mod n \},
$$
then we have the following proposition.

\begin{Prop} \label{prop.HW}
An element $h \in H$ preserves the subvariety $\widehat{\cw}$ if and only if $h \in H_J$.
\end{Prop}

\begin{proof}
An element $h = (\omega^{a_1}, \ldots, \omega^{a_{n-1}})$ of $H$ preserves $\widehat{\cw}$ if and only if $ \omega^{a_j} = 1$ for all $j \in J$.  We may therefore assume $a_j\equiv 0 \mod n$ for all $j\in J$.  Since $h\in H$, it must also satisfy the conditions of~\eqref{e.H}.  The
proposition follows.
\end{proof}

\begin{example} \label{ex.hw} Suppose $n=6$ and consider the partition of $\{1,2,3,4,5\}$ defined by $I = \{1,3,5\}$, $J=\{4\}$, and $K=\{2\}$. Set $(c_j)=(c_4)=(0)$.  In this example, the ideal defining $\hat\cw$ from~\eqref{e.widehat} is $\mathcal{I}= \langle z_2, z_4 - 1 \rangle$.  By Propositions \ref{prop.H} and \ref{prop.HW}, $H$ consists of the set of $(\omega^{a_i})$ satisfying
$a_1 + 2 a_2 + 3 a_3 + 4 a_4 + 5 a_5 = 0$, and $H_J$ is the subgroup of $H$ equal to the set of
$(\omega^{a_i})$ satisfying $a_4 = 0$,
$a_1 + 2 a_2 + 3 a_3 + 5 a_5 \equiv 0 \mod 6$. 
\end{example}

The main result of this subsection, Proposition \ref{prop.quotient} below, is that
there is an isomorphism $\widehat{\cw}/H_J \to \cw$.
We begin with a lemma.

\begin{Lem}\label{lemma.ringiso}  Let $C$ be a ring with an action of a finite group
$M$, and let $I$ be an $M$-stable ideal in $C$.  There is a ring isomorphism $(C/I)^{M} \simeq C^{M} / I^{M}$.
\end{Lem}
\begin{proof} The composition $C^M \to C \to C/I$ has kernel $I^M= I \cap C^M$, 
which yields an injective homomorphism $C^M /I^M \to (C/I)^M$.  To prove that this
homomorphism is surjective, observe that $c+I \in (C/I)^M$ if and only if $c+I= mc+I $ for all $m \in M$.  
Hence $c+I= c'+I$, where 
\[
c' = \frac{1}{|M|}\sum_{m\in M} mc \in C^M,
\]
and surjectivity follows.
\end{proof}

The map $\widehat{\cw}/H_J \to \cw$ is defined as follows.
The ideal of $\widehat{\cw}$ in the ring $A$ is
 \begin{equation} \label{e.widehat}
 \mathcal{I}  =   \langle z_j - \omega^{c_j}, z_k \mid j\in J, k\in K \rangle,
 \end{equation}
 so $\C[\widehat{\cw}] = A/\ci$.  The group $H_J$ preserves $ \mathcal{I}$, and
using Lemma \ref{lemma.ringiso}, we see that
 $$
 \C[\widehat{\cw} / H_J] = (A / \ci)^{H_J} \cong A^{H_J}/ \mathcal{I}^{H_J}.
 $$
The map
$$
\rho: \widehat{\toric} = \spec A \to \toric = \spec A^H
$$
corresponds to the inclusion of rings
$A^H \to A$, and $\cw = \rho(\widehat{\cw})$ is the subvariety of $\toric$
defined by the ideal $\mathcal{I} \cap A^H$.  Thus, 
$\C[\cw] = A/(\mathcal{I} \cap A^H)$.

The composition
\begin{equation}\label{e.ainject0}
A^H \to A^{H_J} \to A^{H_J}/ \mathcal{I}^{H_J}
\end{equation}
has kernel $\mathcal{I} \cap A^H$, and hence induces an injective ring homomorphism 
\begin{equation} \label{e.ainject}
A^H/(\mathcal{I} \cap A^H) \to A^{H_J}/ \mathcal{I}^{H_J}.
\end{equation}
This yields a corresponding map of affine varieties:
\begin{equation}  \label{e.ainject2}
\widehat{\cw}/H_J = \spec A^{H_J}/ \mathcal{I}^{H_J} \to \cw = \spec A^H/(\mathcal{I} \cap A^H).
\end{equation}

We will prove that \eqref{e.ainject2} is an isomorphism by showing that the ring
map  \eqref{e.ainject} is an isomorphism.  We have already observed that  \eqref{e.ainject} is injective,
so it suffices to show surjectivity.  

The next proposition describes the rings $A^H$ and $A^{H_J}$ explicitly as the span of certain monomials in the $z_i$.

\begin{Prop} \label{prop.invariantrings}
\begin{enumerate}
\item $A^H$ is the span of the monomials $z_1^{b_1} \cdots z_{n-1}^{b_{n-1}}$ where, modulo $n$, the
tuple $(b_1, \ldots, b_{n-1})$ is a multiple
of $(1,\ldots, n-1)$.
\item $A^{H_J}$ is the span of the monomials $z_1^{b_1} \cdots z_{n-1}^{b_{n-1}}$ where, modulo $n$, the tuple
$(b_r)_{r \not\in J}$ is a multiple of $(r)_{r \not\in J}$.  \textup{(}The elements $b_j$  for $j \in J$
are arbitrary.\textup{)}
\end{enumerate}
\end{Prop}

\begin{proof} We prove (1).  Since the action of $H$ on $A = \C[\hat\toric]=\C[z_1,\ldots, z_{n-1}]$ takes any monomial to a multiple
of itself, $A^H$ is spanned by the $H$-invariant monomials.  
By Proposition \ref{prop.H}, $H$ consists of the elements $(\omega^{a_1},\dots,\omega^{a_{n-1}})$,
where the $a_i$ satisfy
$\sum_{r = 1}^{n-1} r a_r \equiv 0 \mod n$. 
Rewriting this condition gives 
$$a_{n-1} \equiv a_1 + 2 a_2 + \cdots + (n-2) a_{n-2} \mod n.$$  
So modulo $n$,
$\va = (a_1, \ldots, a_{n-1})$ is in the row span of the $(n-2)\times(n-1)$ matrix,
$$
\begin{bmatrix}
1 & 0 & 0 & \cdots & 0  & 1 \\
0 & 1 & 0 & \cdots & 0  & 2 \\
0 & 0 & 1 & \cdots & 0  & 3 \\
\vdots & \vdots & \vdots & \cdots & \vdots & \vdots \\
0 & 0 & 0 & \cdots & 1  & (n-2)
\end{bmatrix}.
$$
A monomial $z_1^{b_1} \cdots z_{n-1}^{b_{n-1}}$ is therefore $H$-invariant if and only if the dot product of 
$\vb = (b_1, \ldots, b_{n-1})$ with any row of the matrix above is $0$ modulo $n$.  This yields the conditions
\[
b_1 \equiv - b_{n-1} \mod n,\;\; b_2 \equiv -2 b_{n-1} \mod n,\; \ldots, \; b_{n-2} \equiv -(n-2) b_{n-1} \mod n 
\]
so $\vb \equiv -b_{n-1}(1, 2, \ldots, n-1) \mod n$.
This proves (1); the proof of (2) is similar.
\end{proof}

\begin{Prop}\label{prop.invariantsum}
With notation as above, $A^{H_J} = A^H + \mathcal{I}^{H_J}$.
\end{Prop}

\begin{proof}
We will show that given $f = z_1^{b_1} \cdots z_{n-1}^{b_{n-1}} \in A^{H_J}$, we can write $f$
as a sum of an element of $A^H$ and an element of $\mathcal{I}^{H_J}$.  By Proposition~\ref{prop.invariantrings} (2), there exists $a\in \Z$ such that $b_r \equiv ar \mod n$ for all $r\notin J$.
For each $j \in J$, choose a non-negative integer
$m_j$ so that $b_j + m_j \equiv aj \mod n$.  Let $$g = \prod_{j \in J} z_j^{m_j} \cdot f   = \prod_{r \not\in J} z_r^{b_r} \prod_{j \in J} z_j^{b_j + m_j}.$$
Then $g \in A^H$ by Proposition~\ref{prop.invariantrings}(1).  Set $c = \prod_{j \in J} \omega^{c_j m_j}$.  We claim that
$u :=c f -g $ is in $\mathcal{I}$.  Indeed, $u = (c - \prod_{j \in J}z_j^{m_j} ) f $.  Since $c - \prod_{j \in J}z_j^{m_j}$ equals 
$0$ whenever $z_j = \omega^{c_j}$ for all $j\in J$, we have  $c - \prod_{j \in J}z_j^{m_j} \in \mathcal{I}$, proving the claim.
Note that $u \in A^{H_J}$ since $f$ and $g$ are elements of $A^{H_J}$.   Hence $u \in \mathcal{I}^{H_J}$.  We have
$$
f = c^{-1} g + c^{-1} u \in A^H + \mathcal{I}^{H_J},
$$
as desired.  This completes the proof.
\end{proof}

\begin{example}
Suppose $n = 6$, with $\widehat{\cw}$ and $H_J$ as in Example \ref{ex.hw}.  By Proposition~\ref{prop.invariantrings},
$f = z_1 z_2^2 z_3^3 z_4 z_5^5$ is in $A^{H_J}$.  We have $J = \{4 \}$ and $(c_j) = (c_4) = (0)$ from Example \ref{ex.hw}.
Choose $m_4 = 3$ so $g = z_4^3 f$; in the notation of the proof of Proposition \ref{prop.invariantsum}, we have $c = 1$, 
so $h = f - g$.  Since
$1 - z_4^3 \in \mathcal{I}$, we have 
$h = (1 - z_4^3) f \in \mathcal{I}$, and $f = g+h \in A^H + \mathcal{I}^{H_J}$.
\end{example}

The following proposition is the main result of this subsection.

\begin{Prop} \label{prop.quotient}
The natural map $A^H/(\mathcal{I} \cap A^H) \to A^{H_J}/ \mathcal{I}^{H_J}$ of~\eqref{e.ainject} is an isomorphism.  Hence the corresponding map of varieties $\widehat{\cw}/H_J \to \cw$ is an isomorphism.
\end{Prop}

\begin{proof}
As noted above, we need only prove that~\eqref{e.ainject} is surjective, which is equivalent
to the statement that the map $A^H \to A^{H_J}/ \ci^{H_J}$ from~\eqref{e.ainject0} is surjective.
If $f \in A^{H_J}$, then by Proposition~\ref{prop.invariantsum}, we can write $f = f_1 + f_2$, where $f_1 \in A^H$ and $f_2 \in  \mathcal{I}^{H_J}$.
Thus $f + \mathcal{I}^{H_J}$ is the image of $f_1$ under~\eqref{e.ainject0}.  Hence 
the map $A^H \to A^{H_J}/ \ci^{H_J}$ is indeed surjective.
\end{proof}

\begin{example}
The injectivity of the map $A^H/(\mathcal{I} \cap A^H) \to A^{H_J}/ \mathcal{I}^{H_J}$ is valid more generally
for any ring $A$ with an action of a finite group $H$, along with an ideal $\mathcal{I}$ such that $H_J$ is the
subgroup of $H$ preserving $\ci$.  However, in this generality, the map can fail to be surjective.  For example,
let $A = \C[x,y]$, and $H = \Z_2$, with the nontrivial element acting by multiplication by $-1$.  Let
$\mathcal{I} = \langle (x-1)(y-1) \rangle$.  Then $H_J = \{ 1 \}$, so $A^{H_J}/\mathcal{I}^{H_J} = A/\mathcal{I}$.  The map
$$
A^H = \C[x^2, xy, y^2]  \to A/I = \C[x,y]/  \langle (x-1)(y-1) \rangle
$$
is not surjective.  This follows since the corresponding map of affine schemes, which
is the projection of the union of the lines $x=1$ and $y=1$ to
$\C^2/H$, is not injective (e.g.~the points $(1,-1)$ and $(-1,1)$ have the same image).
\end{example}


\subsection{Components and the action of $Z$} \label{ss.componentsZ}
Recall the diagram:
\begin{equation} \label{e.toricmaps2}
\begin{tikzcd} 
\widehat \toric \cong \C^{n-1} \arrow[rr,bend left,"\widehat \rho"] \arrow[r] \arrow[r, "\rho"] & \toric \arrow[r,"\pi"] & \toric_{ad} \cong \C^{n-1}.
\end{tikzcd}
\end{equation}
The main results of this subsection, Propositions \ref{p.components} and \ref{p.zaction}, give a parametrization of the set of
components of $\pi^{-1}(\cw_{ad})$ in $\toric$, and describe the $Z$-action on this set of components
in terms of the parametrization.

Write $J^c := I \cup K$ for the complement of $J$ in $[n-1]$.  The ideal of $\cw_{ad}$ in $\C[\toric_{ad}] = \C[x_1,\ldots, x_{n-1}]$ is $\mathcal{I}(\cw_{ad}) = \langle x_j - 1, x_k \mid j\in J,\,k\in K \rangle$.  The inverse image $\widehat{\rho}^{-1}(\cw_{ad})$ is a disjoint union of subvarieties $\hat\cw_{(c_j)}$ as defined in the previous section, one for each tuple of integers
 $(c_j)\in \Z^{|J|}$ with $c_j\in \{0,\ldots, n-1\}$.  Each $\hat\cw_{(c_j)}$ is isomorphic to affine space $\C^{|I|}$, and
 there are $|J|^n$ such components, each corresponding to a tuple $(c_j)$.  The group $\widehat{H} = \ker(\hat\rho)\simeq Z_n \times \cdots \times \Z_n$ acts transitively on the set of these components.  Recall that $\cw_{(c_j)}:= \rho(\hat\cw_{(c_j)}) \subseteq \toric$.
 
\begin{Lem} \label{lem.disjoint}
Let $(c_j),(d_j)\in \Z^{|J|}$ with $c_j,d_j\in \{0,\ldots, n-1\}$.  Then either $\cw_{(c_j)} = {\cw}_{(d_j)}$, or
${\cw}_{(c_j)}$ is disjoint from $\cw_{(d_j)}$.
\end{Lem}

\begin{proof}
We prove that if $\cw_{(c_j)}$ and $\cw_{(d_j)}$ intersect, then they coincide.
Suppose then that these intersect.  Since $\rho$ is the quotient by $H$, this means that there is some
$p \in \widehat{\cw}_{(c_j)}$ and some $h \in H$ such that $h\cdot p\in \hat\cw_{(d_j)} $. But this implies
that $h \cdot \widehat{\cw}_{(c_j)} = \widehat{\cw}_{(d_j)}$, since 
as noted above, the group $\widehat{H}$ permutes the components of $\widehat{\rho}^{-1}(\cw_{ad})$.  Thus $\cw_{(c_j)} = \cw_{(d_j)}$.
\end{proof}
 
 It follows that $\pi^{-1}(\cw_{ad})$ is the disjoint union of the distinct ${\cw}_{(c_j)}$, and that these are the irreducible
 components of $\pi^{-1}(\cw_{ad})$.

 Let $d_* := \gcd( r\mid r\in [n]\setminus J = J^c\cup\{n\} )$.  In other words, $d_*$ is a generator of the
subgroup of $\Z_n$ generated by the elements $r\in [n]\setminus J$. 
Define $\phi:  \Z^{|J|} \to \Z_{d_*}$ by  $\phi((c_j) ) = \sum_{j \in J} j c_j  \mod d_*$.  We abuse
notation and write $\phi(\cw_{(c_j)}) = \phi((c_j))$.

\begin{Prop} \label{p.components} 
The components $\widehat{\cw}_{(c_j)}$ and $\widehat{\cw}_{(d_j)}$ of $\hat\rho^{-1}(\cw_{ad})$ are in the same $H$-orbit
if and only if $\phi((c_j)) = \phi((d_j))$.  The map $\phi : \{ \cw_{(c_j)} \mid (c_j)\in \Z^{|J|}  \} \to \Z_{d_*}$ induces a bijection between 
the components of $\pi^{-1}(\cw_{ad})$ and $\Z_{d_*}$.
\end{Prop}

\begin{proof}
Let $h = (\omega^{a_j})_{j \in [n-1]} \in \widehat{H}$.  Then 
\begin{equation} \label{e.horbit}
h \cdot \widehat{\cw}_{(c_j)} = \widehat{\cw}_{(d_j)}
\end{equation}
 if and only
if $a_j \equiv d_j - c_j\mod n$ for all $j \in J$.  By Proposition~\ref{prop.H}, the element $h$ is in $H$ if and only if $\sum_{r=1}^{n-1} r a_r \equiv 0\mod n$.  Thus, there exists $h\in H$ satisfying~\eqref{e.horbit} if and only if we can find a tuple $(a_r)_{r \in J^c}$ satisfying the equation
\begin{eqnarray}\label{eqn.components}
\sum_{r \in  J^c} r a_r \equiv - \sum_{j \in J} j(d_j - c_j )\mod n.
\end{eqnarray}
The left side of~\eqref{eqn.components} is divisible by $d_*$, as is $n$.  Therefore, the existence of such a tuple implies 
that $\sum_{j \in J} j c_j  \equiv \sum_{j \in J} j d_j \mod d_*$, that is, $\phi((c_j)) = \phi((d_j))$.  Hence, if $\widehat{\cw}_{(c_j)}$
and $\widehat{\cw}_{(d_j)}$ are in the same $H$-orbit, then $\phi((c_j)) = \phi((d_j))$.  

Conversely, suppose
$\phi((c_j)) = \phi((d_j))$.  We will show there is a tuple $(a_r)_{r \in J^c}$ satisfying~\eqref{eqn.components}. 
As $d_*=\gcd(r\mid r\in J^c \cup \{n\})$, B\'ezout's identity implies that
$\sum_{r\in J^c} rb_r \equiv d_* \mod n$ for some $b_r\in \Z$.  Since $\sum_{j \in J} j c_j  \equiv \sum_{j \in J} j d_j \mod d_*$,
we have $- \sum_{j \in J} j(d_j - c_j )=\tau d_*$ for some $\tau\in \Z$.  For each $r\in J^c$, set $a_r = \tau b_r$.  The tuple $(a_r)_{r\in J^c}$ then satisfies~\eqref{eqn.components}.  Hence $\widehat{\cw}_{(c_j)}$
and $\widehat{\cw}_{(d_j)}$ are in the same $H$-orbit.  This proves the first assertion of the proposition, and 
shows that $\phi$ gives an injective map from the set of components of $\pi^{-1}(\cw_{ad})$ to $\Z_{d_*}$.  

To complete the proof, we must show that $\phi$ is surjective.  For this,
we need only show that the $j \in J$ generate $\Z_{d_*}$.  This is equivalent to showing that the $j \in J$ together
with $d_*$ generate $\Z$.  This holds since $d_* \Z$ is the subgroup of $\Z$ generated by $[n] \setminus J$.
\end{proof}

\begin{Rem} \label{rem:non-scheme}
When we refer to the inverse image of a variety, we mean the inverse image with its reduced scheme
structure.  Indeed, $\pi^{-1}(\cw_{ad})$ may be nonreduced if given the
inverse image scheme structure.  For example, suppose $n= 4$,
$I$ is empty, $J = \{1,3 \}$ and $K = \{ 2 \}$.  We have $\toric_{ad} = \spec \C[x_1, x_2, x_3] \cong \C^3$,
and $\cw_{ad}$ is the point $(1,0,1)$.  We have
$\toric = \spec B$, where $B$ is the quotient of the polynomial ring $\C[x_1,x_2,x_3,y_1, y_2, y_3]$
by an ideal
$P$ containing the elements
$y_1^4 - x_1^3 x_2^2 x_3$, $y_2^2 - x_1 x_3$, and $y_3^4 - x_1 x_2^2 x_3^3$.  The
inverse image of $\cw_{ad}$ is the subscheme in $\toric$ corresponding to the ideal $Q$ in $B$ generated by
$x_1-1, x_2$, and $x_3-1$.  As a reduced scheme, $\pi^{-1}(\cw_{ad})$ is a union of two points,
where $x_1 = x_3 = 1$, $x_2 = y_1 = y_3 = 0$, and $y_2 = \pm 1$.  However, the scheme-theoretic
inverse image $\pi^{-1}(\cw_{ad})$ is not reduced.  Indeed, $y_1^4$ and $y_3^4$ are in $Q$, but we claim $y_1$ and $y_3$ are not.
We sketch the verification of this claim for $y_1$.  Suppose by contradiction that $y_1 \in Q$; then $y_1 = (x_1-1)b_1 + x_2 b_2 + (x_3-1) b_3$ for $b_i \in B$.  
The center $Z$ acts trivially on the $x_i$ and on $y_1$ by the character $e^{-\mu_1}|_Z$,
so by replacing each $b_j$
with the appropriate $Z$-isotypic component, we may assume that $Z$ acts on each
$b_j$ by $e^{-\mu_1}|_Z$.  But any element of $B$ on which $Z$ acts by $e^{-\mu_1}|_Z$
is divisible by $y_1$.  Therefore, $b_j = y_1 c_j$ for $c_j \in B$, so
$y_1(1 - (x_1-1)c_1 - x_2 c_2 - (x_3-1) c_3) = 0$.  Since $B$ is an integral domain,
$(x_1-1)c_1 + x_2 c_2 + (x_3-1) c_3 = 1$, but this is impossible, as the elements $x_1-1, x_2, x_3 - 1$ do not generate
the unit ideal in $B$ (since they do not generate the unit ideal in the larger ring $A$).
We conclude that $y_1 \not\in Q$, as claimed.
\end{Rem}

Given $r \in \Z_{d_*}$, we denote by $\cw_r$ the component $\cw_{(c_j)}$ for any tuple $(c_j)$ with $\phi((c_j)) = r$.
This component depends on the decomposition $[n-1] = I \sqcup J \sqcup K$, but we omit this from the notation.

The group $Z = \hat{H}/H$ acts on $\toric$.  Recall that $\underline{\omega}$ denotes the generator of $Z$.

\begin{Prop} \label{p.zaction}
With notation as above, we have $\underline{\omega} \cdot \cw_r = \cw_{r+1}$.
\end{Prop}

\begin{proof}
If $1 \notin J$, then $1 \in J^c$ so $d_* = 1$, the group $\Z_{d_*}$ is trivial, and the proposition holds.  Therefore we assume $1 \in J$. The element $\hat{h} \in \widehat{H}$
defined by 
$\hat{h} = (\omega, 1, 1, \ldots, 1)$ is a lift of $\underline{\omega} \in Z$ to $\widehat{H}$, in the sense that the map \eqref{e.hmap}
takes $\hat{h}$ to $\underline{\omega}$.  Suppose $\phi((c_j)) = r$.   The action of $Z$ on the component $\cw_{(c_j)}$ is defined by
$$
\underline{\omega} \cdot \cw_{(c_j)} = \underline{\omega} \cdot \rho(\widehat{\cw}_{(c_j)}) = \rho (\hat{h} \cdot \widehat{\cw}_{(c_j)}).
$$
We have $\hat{h} \cdot \widehat{\cw}_{(c_j)} = \widehat{\cw}_{(d_j)}$ where the tuple $(d_j)_{j\in J}$ is defined by $d_1 = c_1 + 1$
and $d_j = c_j$ for $j \neq 1$.  Thus $\underline{\omega}\cdot \cw_{(c_j)} = \cw_{(d_j)}$ where
$$
 \phi ((d_j))  =  \sum_{j\in J} j d_j = 1  + \sum_{j\in J} jc_j \equiv (r +1) \mod d_*.
$$
The result follows.
\end{proof}

If $V$ is a vector space with basis $\{\cw_r \mid r\in \Z_{d_*} \}$, and we define a representation of $Z$ on $V$ by the rule $\underline{\omega} \cdot \cw_r = \cw_{r+1}$, then the matrix of $\underline{\omega}$ with respect to this basis is 
$$
\begin{bmatrix}
0  & 0 & 0 & \cdots & 0 & 1 \\
1 & 0 & 0 & \cdots &0 & 0 \\
0 & 1 & 0 & \cdots & 0 & 0 \\
\vdots & \vdots & \vdots & \ddots &  \vdots & \vdots \\
0 & 0 & 0 & \cdots &0 & 0\\
0  & 0 & 0 & \cdots & 1 & 0
\end{bmatrix}.
$$
This matrix has eigenvalues $1, \omega^q, \omega^{2q}, \ldots, \omega^{(d_* -1) q}$, where $q = n/d_*$.
(Note that $\omega^q$ is an $d_*$-th root of $1$.)  Thus we see that $V$ decomposes under $Z$ as a direct sum
of the $1$-dimensional representations with characters $1, \chi_q, \chi_{2q}, \ldots, \chi_{(d_* -1) q}$, where
$\chi_p$ is the character of $Z$ satisfying $\chi_p(\underline{\omega}) = \omega^p$.  We have obtained the following corollary.

\begin{Cor}\label{cor.Zaction} The character of the $Z\cong \Z_n$ representation on the vector space 
spanned by the components of $\pi^{-1}(\cw_{ad})$ is $1+ \chi_q+ \chi_{2q}+ \cdots+ \chi_{(d_* -1) q}$, where $q=n/d_*$.~$\Box$
\end{Cor}

We now introduce some notation which will be useful later in the paper, when we apply the results of this section.  
By definition, $\widehat{\cw} = \widehat{\cw}_{(c_j)} \cong  \C^{|I|}$ is a subvariety
of $\hat{\cv}$ depending on a tuple $(c_j)_{j\in J}$ of integers.  There is an $H_J$-equivariant isomorphism  
$\widehat{\cw}_{(c_j)} \simeq  \widehat{\cw}_{(d_j)}$,
which is the identity on the coordinates $z_{\ell}$ ($\ell \not\in J$) and which changes the value of the coordinate $z_j$ from $\omega^{c_j}$ to $\omega^{d_j}$. 
This induces an isomorphism on the quotients by $H_J$, which
we use to identify $\widehat{\cw}_{(c_j)} / H_J $ with  $\widehat{\cw}_{(d_j)} / H_J$.  
We set $\tilde{\C}^{|I|} := \widehat{\cw}_{(c_j)} / H_J$; by the remarks above, 
$\tilde{\C}^{|I|}$ is independent of the choice of tuple $(c_j)$.   Define
$i_r: \tilde{\C}^{|I|} \to \cw_r$ to be the isomorphism of Proposition \ref{prop.quotient}.  
By abuse of notation, we view $i_r$ as a map $\tilde{\C}^{|I|} \to \toric$ with image $\cw_r$.

Viewing the action of the generator $\underline{\omega}$ of $Z$ as a map from $\toric$ to itself, we restate
Proposition \ref{p.zaction} in terms of the maps $i_r$.

\begin{Prop} \label{p.zaction2} 
For all $r\in \Z_{d_*}$, we have $\underline{\omega} \circ i_r = i_{r+1}$.
\end{Prop}

\begin{proof}
Suppose $\phi((c_j)) = r$.  Let $\hat{h}$ and $(d_j)$ be as in the proof of Proposition \ref{p.zaction}.  
The action on $\hat{h}$ on $\hat{\cv}$ induces an isomorphism $\widehat{\cw}_{(c_j)} /H_J \xrightarrow{\; \simeq \;}  \widehat{\cw}_{(d_j)} /H_J$.
We obtain the diagram
\begin{equation} \label{e.commdiag-zaction2}
\begin{CD}
\tilde{\C}^{|I|} = \widehat{\cw}_{(c_j)} /H_J @>{i_r}>>   \toric  \\
@V\hat{h}VV                             @V{\underline{\omega}}VV\\
\tilde{\C}^{|I|} =  \widehat{\cw}_{(d_j)}/H_J @>{i_{r+1}}>>    \toric .
\end{CD}
\end{equation}
In this diagram, we have identified both $\widehat{\cw}_{(c_j)} /H_J$ and $\widehat{\cw}_{(c_j)} /H_J$ with $\tilde{\C}^{|I|}$ as above.
The proof of Proposition \ref{p.zaction} shows that the action of $\underline{\omega}$ on $\toric$ is induced by the action
 of $\hat{h}$ on $\hat{\toric}$, and that this action takes $\cw_r$ to $\cw_{r+1}$.
Since by definition $i_r$ and $i_{r+1}$ are the isomorphisms
 of $\widehat{\cw}_{(c_j)} /H_J$ and $\widehat{\cw}_{(d_j)}/H_J$ with $\cw_r$ and $\cw_{r+1}$,
 respectively, the diagram commutes.  The proposition follows.
\end{proof}

Let
$q:  \tilde{\C}^{|I|} \to \C^{|I|}$ be the map making the following diagram commute: 
\begin{equation} \label{e.commdiag}
\begin{CD}
 \tilde{\C}^{|I|} @>{i_r}>>   \cw_r  \\
 @VqVV                             @V{\pi}VV\\
 \C^{|I|} @>i>>    \cw_{\ad} .
\end{CD}
\end{equation}
Note that $q$ is uniquely defined since the horizontal maps
are isomorphisms.

\subsection{Rescaling Argument}

Our study of the components of $\pi^{-1}(\cw_{ad})$ above utilizes the action of groups $H$ and $\hat{H}$ on $\rho^{-1}(\cw_{ad})$.  One can also study the subvarieties $\cw_{(c_j)} = \rho(\hat\cw_{(c_j)})$ directly, using coordinates on $\toric$ and $\hat{\toric}$.  This provides an alternative approach
to some of the results in Section \ref{ss.componentsZ}.  In this subsection we give an example illustrating this concrete, albeit technical, approach.  The results of this subsection are not
needed for the remainder of the paper.

\begin{example} Let $n=12$ and consider the set partition of $[11]$ defined by $I = \{10\}$, $K=\{4,8\}$, and $J = [11]\setminus \{4,8,10\}$.  We have $d_* = \gcd(4,8,10,12)=2$.
The results earlier in this section tell us that there are precisely two components of $\pi^{-1}(\cw_{ad})$, each indexed by an element of $\Z_2=\{0,1\}$.  

To see this directly using coordinates, recall that $\C[\toric]$ is a quotient of the ring $\C[v_1,v_2, \ldots, v_{11}, x_1, x_2,\cdots x_{11}]$, where $v_i:=e^{\mu_i}$.  Here,
$\mu_i = \sum c_j \alpha_j$ is the unique weight in the same coset of $\lambda_i$ modulo the root lattice 
which satisfies $0 \leq c_i < 1$ for each $i$.  For example, using~\eqref{e.glk}, we see that
\[
12\lambda_4 = 8\alpha_1 + 16 \alpha_2 + 24 \alpha_3+32 \alpha_4+28\alpha_5+24\alpha_6+20\alpha_7+16\alpha_8+9\alpha_9+6\alpha_{10}+3\alpha_{11}
\]
and thus
\[
\mu_4 = \frac{1}{3}\left(2\alpha_1 + \alpha_2+2\alpha_4+\alpha_5+2\alpha_7+\alpha_8+2\alpha_{10}+\alpha_{11} \right);
\]
see~\cite[Ex.~3.8]{GPR}.  Note that we can view each $v_i$ as a function on $\hat{\toric}$ via pullback
(cf.~\eqref{e.inclusions}), and we adopt that convention for the rest of this section.  Since $n=12$, we have $x_i = e^{\alpha_i} = z_i^{12}$, so
the equation above for $\mu_4$ tells us that, as a function on $\hat{\toric}$, $v_4 = z_1^8z_2^4z_4^8z_5^4 z_7^8z_8^4 z_{10}^8 z_{11}^4$.  Similar computations show the following.
\begin{align*}
v_1 &= z_1^{11}z_2^{10}z_3^{9}z_4^8z_5^{7}z_6^{6}z_7^{5}z_8^{4}z_9^{3}z_{10}^{2}z_{11}^{\empty}  & v_7 &=  z_1^{5}z_2^{10}z_3^{3}z_4^8z_5^{1}z_6^{6}z_7^{11}z_8^{4}z_9^{9}z_{10}^{2}z_{11}^{7} \\
v_2 &= z_1^{10}z_2^{8}z_3^{6}z_4^{4}z_5^{2} z_7^{10}z_8^{8}z_9^{6}z_{10}^{4}z_{11}^{2}  & v_8&= z_1^4z_2^8z_4^4z_5^8 z_7^4z_8^8 z_{10}^4 z_{11}^8 \\
v_3 &= z_1^{9}z_2^{6}z_3^{3} z_5^{9} z_6^{6}z_7^{3} z_9^{9}z_{10}^{6}z_{11}^{3} & v_9&= z_1^{3}z_2^{6}z_3^{9} z_5^{3} z_6^{6}z_7^{9} z_9^{3}z_{10}^{6}z_{11}^{9}   \\
v_4 &=z_1^8z_2^4z_4^8z_5^4 z_7^8z_8^4 z_{10}^8 z_{11}^4 & v_{10}&= z_1^{2}z_2^{4}z_3^{6}z_4^{8}z_5^{10} z_7^{2}z_8^{4}z_9^{6}z_{10}^{8}z_{11}^{10}  \\
v_5 &= z_1^{7}z_2^{2}z_3^{9}z_4^{4}z_5^{11}z_6^{6}z_7^{\empty}z_8^{8}z_9^{3}z_{10}^{10}z_{11}^{5} & v_{11}&= z_1^{\empty}z_2^{2}z_3^{3}z_4^4 z_5^{5}z_6^{6}z_7^{7}z_8^{8}z_9^{9}z_{10}^{10}z_{11}^{11}  \\
v_6 &= z_1^{6}z_3^{6}z_5^{6}z_7^{6}z_9^{6}z_{11}^{6}
\end{align*}
\end{example}
For each tuple $(c_j)_{j\in J}$, the affine variety $\hat{\cw}_{(c_j)}$ in $\toric$ has defining ideal $\ci(\hat{\cw}_{(c_j)}) = \left< z_j-\omega^{c_j} \mid j\in J \right>+ \left< z_{4}, z_8 \right>$.  Thus $\ci(\hat{\cw}_{(c_j)})$ contains $v_1$, $v_2$, $v_4$, $v_5$, $v_7$, $v_8$, $v_{10}$, $v_{11}$, as each of these has $z_4$ and $z_8$ as factors. (In general, $z_k$ with $k\in K$ will appear as a factor of $v_i$ if and only if $i$ is not divisible by $\frac{n}{\gcd(k,n)}$; see~\cite[Lemma 3.10]{GPR}.)  
In $\C[\hat{\cw}_{(c_j)}]$, we have $z_j = \omega^{c_j}$ for all $j\in J$, so 
\[
v_3 = \omega^{3\Theta}z_{10}^6, \quad v_6 = \omega^{6\Sigma}, \quad v_9 = \omega^{9\Theta}z_{10}^6
\]
where $\Theta = 3c_1+2c_2+c_3+3c_5+2c_6+c_7+3c_9+ c_{11}$ and $\Sigma = c_1+c_3+c_5+c_7+c_9+c_{11}$.  
The equation for $v_6$ implies that the $v_6$-coordinate of any point $p \in \cw_{(c_j)}$ is a constant depending only on the parity of $\Sigma$.  In particular,  
\[
v_6(p) = \left\{\begin{array}{lc} 1  & \textup{ if $\Sigma$ is even} \\ \omega^6=-1 & \textup{ if $\Sigma$ is odd} \end{array} \right.
\]
for all $p\in \cw_{(c_j)}$.  This shows $\pi^{-1}(\cw_{ad})$ has at least two components, each determined by the parity of the sum $\Sigma$ for a given tuple $(c_j)_{j\in J}$.

To prove that there are precisely two components, one must show that the $v_6$-coordinate uniquely determines the image of $\hat{\cw}_{(c_j)}$ under $\rho$.  Indeed, suppose $(c_j')_{j\in J}$ is another tuple such that $\Sigma'=c_1'+c_3'+c_5'+c_7'+c_9'+c_{11}'$ has the same parity as $\Sigma$.  Set $\Theta'=3c_1'+2c_2'+c_3'+3c_5'+2c_6'+c_7'+3c_9'+ c_{11}'$.  Observe that the maps
$\hat{\cw}_{(c_j)} \to \C$, $p \mapsto z_{10}(p)$, and $\hat{\cw}_{(c_j')} \to \C$, $p \mapsto z_{10}(p')$, are isomorphisms.
We obtain an isomorphism $\hat{\cw}_{(c_j)} \to \hat{\cw}_{(c_j')}$,
 $p \mapsto p'$, characterized by the relation $z_{10}(p) =  z_{10}(p')$. 
Any $p\in \hat{\cw}_{(c_j)}$ with $z_{10}(p)=b_{10}$ satisfies 
\[
v_3(p) = \omega^{3\Theta}b_{10}^6 \quad \textup{ and } \quad v_9(p)=\omega^{9\Theta} b_{10}^6
\]
while $p'\in \hat{\cw}_{(c_j')}$ such that $z_{10}(p')=b_{10}$ satisfies
\[
v_3(p') = \omega^{3\Theta'}b_{10}^6 \quad \textup{ and } \quad v_9(p')=\omega^{9\Theta'} b_{10}^6.
\]
Note that $\Sigma$ and $\Sigma'$ have the same parity if and only if the same is true of $\Theta$ and $\Theta'$ so we may assume $\Theta-\Theta'$ is even.
To show that $\cw_{(c_j)} = \cw_{(c_j')}$ consider the ``rescaling map'' 
\begin{eqnarray}\label{eqn.rescaling}
\tau: \C  \to \C;\; b_{10} \mapsto \omega^{\frac{1}{2}\left( \Theta-\Theta' \right)}b_{10}.
\end{eqnarray}
Using the isomorphism $\C\cong \hat{\cw}_{(c_j')}$, we view the rescaling map as an automorphism of $\hat{\cw}_{(c_j')}$.
Under this rescaling, for all $p'\in \hat{\cw}_{(c_j')}$, we obtain
\[
v_3\circ \tau(p') =  \omega^{3\Theta'}\omega^{3\left( \Theta - \Theta' \right)}b_{10}^6 = \omega^{3\Theta}b_{10}^6 = v_3(p).
\]
Similarly, $v_9 \circ \tau(p')= v_9(p)$.  Hence the images of $\hat{\cw}_{(c_j)}$ and $\hat{\cw}_{(c_j')}$ are equal under $\rho$, as desired.

\

The program outlined in the example can be carried out more generally to prove the results of Proposition~\ref{p.components} and Corollary~\ref{cor.Zaction} using the functions $z_i$ and $v_i$ on $\hat{\toric}$.  The key point is that, when viewed as a function on $\hat{\toric}$ via pull-back, the $v_{{n}/{d_*}}$-coordinate of any point in $\hat{\cw}_{(c_j)}$ will be an element of $\Z_{d_*} \cong \{ \omega^{n/d_*}, \omega^{2n/d_*}, \ldots, 1 \}$.  It is then possible to define a ``rescaling map'' as in~\eqref{eqn.rescaling} in order to prove
that the image $\rho(\hat{\cw}_{(c_j)})=\cw_{(c_j)}$ is uniquely determined by the value of its $v_{n/d_*}$-coordinate, thereby recovering Proposition~\ref{p.components}.  Finally, the $Z$-action on the components of $\pi^{-1}(\cw_{ad})$ is determined by the $Z$-action on the $v_{n/d_*}$-coordinate and given explicitly by multiplication by $\omega^{n/d_*}$, so we recover the results of Corollary~\ref{cor.Zaction}.


\section{Affine Pavings} \label{sec.paving}

Our goal is to prove that each extended Springer fiber has an affine paving modulo a finite group action. 
We begin with some combinatorial data.

\subsection{Row-Strict tableaux} \label{ss.row-strict}

In the previous section, our analysis relied on disjoint subsets $I$, $J$, and $K$ partitioning the set $[n-1] =\{1,2,\dots, n-1\}$.  In this section, we obtain these sets from the row-strict tableaux which parametrize pieces of the affine paving for each Springer fiber.

Fix a positive integer $n$ and a partition $\lambda$ of $n$. The Young diagram of shape $\lambda$ is a collection of boxes arranged into rows and columns corresponding to the parts in the partition $\lambda$.  We orient our diagrams using the English convention, so the rows are left justified with sizes decreasing from top to bottom.  A tableau of shape $\lambda$ and content $[n]$ is obtained by filling each box with one of the numbers $1$ through $n$, with no repetition.  The \textit{base filling} of $\lambda$ is the tableau obtained by filling the boxes of $\lambda$ with the the numbers $1$ through $n$ by
starting at the bottom of the leftmost column and moving up the column by increments of $1$, then moving to the lowest
box of the next column to the right, and so on.  For example, the base filling of $\lambda = [4\ 3\ 1]$ is the following.
\[\ytableausetup{centertableaux} \begin{ytableau}3 & 5 & 7 & 8\\ 2 & 4 & 6\\ 1\end{ytableau}
\]

Tableaux such that the entries increase across each row are called \textit{row-strict}. Denote the set of row-strict tableaux of shape $\lambda$ by $\RST(\lambda)$.

\begin{figure}[h]
\begin{ytableau}3 & 4 & 5 & 6\\ 1 & 2 & 7\\ 8\end{ytableau}\quad\quad \begin{ytableau}1 & 2 & 3 & 4\\ 5 & 6 & 7\\ 8\end{ytableau}\quad\quad \begin{ytableau}5 & 6 & 7 & 8\\ 1 & 2 & 4\\ 3\end{ytableau}

\caption{Let $n=8$. Examples of row-strict tableaux of shape $\lambda = [4\ 3\ 1]$.}\label{fig.tableaux}
\end{figure}

Let $\sigma \in \RST(\lambda)$.  We define subsets $I_\sigma$, $J_\sigma$, and $K_\sigma$ of $[n-1]$ as follows:
\begin{eqnarray*}
I_\sigma &:=& \left\{ i\in [n-1] \mid \begin{array}{c} \textup{if $i+1$ is in the column directly to the right of $i$} \\ \textup{and in any higher row or $i+1$ is in}\\ \textup{any column at least two to the right of $i$ }  \end{array} \right\} \\
J_\sigma &:=& \{i\in [n-1]\mid i+1 \text{ is in the same row as $i$}\}\\
K_\sigma &:=& [n-1] \setminus (I_\sigma \sqcup J_\sigma).
\end{eqnarray*}
If $i \in I_\sigma$, write $\ell(i+1)$ for the entry to the left of $i+1$.  If $\sigma$ is fixed, we may omit the subscript $\sigma$ and simply denote these sets by
$I, J$ and $K$.

By definition, the position of $i+1$ in $\sigma$ determines which subset contains $i$.  The diagram in Figure~\ref{fig.subsets} below illustrates this definition. 

\begin{figure}[h]
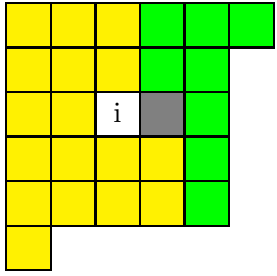

\ydiagram[*(white) $i$]
  {0,0,2+1,0,0}
  *[*(gray)]{0,0,3+1,0,0,0}
  *[*(yellow) ]{3,3,4,4,4,1}
*[*(green)]{6,5,5,5,5,1}
\caption{If $i+1$ is in one of the green boxes, then $i\in I_\sigma$.  If $i+1$ is in the gray box, then $i\in J_\sigma$.  If $i+1$ is in one of the yellow boxes, then $i\in K_\sigma$.}\label{fig.subsets}
\end{figure}

\begin{example}\label{ex.IJK}
Let $n=12$ and $\lambda = [4^2\ 2^2]$. Let $\sigma\in \RST([4^2\ 2^2])$ be the tableau pictured below.
\[\ytableausetup{centertableaux} \begin{ytableau}3 & 4 & 5 & 6\\ 1 & 2 & 9 & 10\\ 7 & 8\\ 11 & 12\end{ytableau}
\]
Then $I_\sigma = \{8\}$, $J_\sigma = \{1,3,4,5,7,9,11\}$, and $K_\sigma = \{2,6,10\}$.
\end{example}

\begin{Def} \label{def.Springer.inv} Let $i,j\in [n]$ such that $i>j$.  We say that the pair $(i,j)$ is a \textit{Springer inversion} of $\sigma\in \RST(\lambda)$ if 
\begin{enumerate}
\item $i$ occurs in a box below $j$ and in the same column or in any column strictly to the left of the column containing $j$ in $\sigma$, and
\item if the box directly to the right of $j$ in $\sigma$ is labeled by $r$, then $i< r$.
\end{enumerate}
Let $|\sigma|$ be the number of Springer inversions of $\sigma\in \RST(\lambda)$.  We denote the set of Springer inversions by $\inv(\sigma)$, so $|\sigma| := |\inv(\sigma)|$.  
\end{Def}
Note that if $i\in I_\sigma$, then $(i,\ell)$ is a Springer inversion, where $\ell = \ell(i+1)$ denotes the entry in the box of $\sigma$ directly to the left of the box containing $i+1$.  For example, in Example \ref{ex.IJK}, $(8,2)$ is a Springer inversion.  Let $\tilde{I}_{\sigma}$ denote the set of inversions of the form $(i, \ell(i+1))$, for $i \in I_{\sigma}$.  Since $\tilde{I}_{\sigma} \subset \inv(\sigma)$, we have $|I_{\sigma}| = |\tilde{I}_\sigma| \leq |\sigma|$.

\begin{example}\label{ex.inversions} If $\sigma$ is the row-strict tableau of shape $[4^2\ 2^2]$ appearing in Example~\ref{ex.IJK},
we have $I_{\sigma} = \{ 8 \}$ and, as noted above, $(8, 2)\in \inv(\sigma)$.  The interested reader can check that 
\begin{eqnarray*}
\inv(\sigma) &=& \{ (12, 8), (12,10), (12, 6), (11,8), (11, 10), (11, 6),\\
&&\quad \quad\quad\quad (10, 6), (9,6), (8, 2), (8,6) (7,2), (7,6), (3,2) \}
\end{eqnarray*}
so $|\sigma|=13$ in this case.  
\end{example}

\begin{Rem}
The Springer inversions form a subset of the inversion set of a particular permutation $w_{\sigma}$. 
Here, our convention is that the inversion set of a permutation $w$ is the set of pairs $(i,j)$ with $i>j$ such that $w(i)<w(j)$ (that is, in
the $1$-line notation for $w^{-1}$, the number $i$ occurs before $j$).  Define $w_{\sigma}$
to be the permutation of $[n]$ such that if $p$ occurs in the box of $\sigma$ in which $q$ occurs in the base 
filling of $\lambda$, then $w_{\sigma}(p) = q$.   In other words, order the boxes of $\sigma$ as in the base filling,
and list the entries in $\sigma$ in the order in which the boxes occur.  This list is the $1$-line notation
for $w_{\sigma}^{-1}$.  For example, if $\sigma$ is as in Example \ref{ex.IJK}, 
then $w_{\sigma}^{-1} = [11,7,1,3,12,8,2,4,9,5,6,10]$.
The pairs $(i,j)$ satisfying condition (1) in Definition \ref{def.Springer.inv}
 are the inversions of $w_{\sigma}^{-1}$, so the Springer inversions form a subset of these.
  In this example, $w_{\sigma}^{-1}$ has $28$ inversions, but there are only $13$ Springer inversions.  
  The permutations $w_{\sigma}$ are related to the paving of a Springer fiber discussed below.  Indeed,
the pieces of the paving are obtained by intersecting the Springer
fiber with the Schubert cells indexed by the 
$w_{\sigma}$; see~\cite[Theorem 5.4]{Ji-Precup2022}.
\end{Rem}

\begin{Def}\label{def.divisor} Given a tableau $\sigma\in \RST(\gl)$, a \textit{block} in $\sigma$ is a set of consecutive boxes in a single row of $\sigma$ labeled by consecutive numbers.  If $\sigma$ can be broken up into blocks of size $d\in \N$ then we say that \textit{$d$ is a divisor of $\sigma$}.  For each $\sigma\in \RST(\gl)$ we let $d_\sigma$ denote its maximal divisor.
\end{Def}

The tableau pictured in Example~\ref{ex.IJK} has maximal divisor equal to $2$.

\begin{Lem}\label{lem.maxdiv} For each $\sigma\in \RST(\gl)$, 
\[
d_\sigma = \gcd(i \mid i \in K_\sigma \sqcup I_\sigma \sqcup \{ n \}).
\]
\end{Lem}
\begin{proof} 
Suppose that 
$$K_\sigma \sqcup I_\sigma = \{ n_1, n_2, \ldots, n_r \}.$$
By definition of $J_\sigma$, the tableau $\sigma$ is then composed of blocks as follows: the first block is  $1, \ldots, n_1$, the second block is
$n_1 + 1, \ldots, n_2$, etc.  The lengths of the blocks are $n_1, n_2 -n_1, \ldots, n_r - n_{r-1}, n - n_r $.
The maximal divisor $d_\sigma$ is the greatest common divisor of the lengths of the blocks, which equals
$\gcd( n_1, n_2, \ldots, n_r, n )$.
\end{proof}

\begin{example}\label{ex.642} Let $n=12$ and consider the row-strict tableau $\sigma\in \RST([6 \ 4 \ 2])$ below. 
\[
\ytableausetup{centertableaux} \begin{ytableau}1 & 2 & 3 & 4 & 11 & 12\\ 5 & 6 & 7 & 8\\ 9 & 10\end{ytableau}
\]
We have $I_\sigma = \{ 10 \}$ and $J_\sigma = \{ 1, 2, 3, 5, 6, 7, 9, 11 \}$.  Now $I_\sigma \sqcup K_\sigma \sqcup \{12\} = \{ 4, 8, 10, 12 \}$ and thus $d_\sigma =2$ by Lemma~\ref{lem.maxdiv}, which is also obvious from the picture of $\sigma$ above.
\end{example}


\subsection{A paving of the Springer fiber}

The fact that Springer fibers are paved by affines was originally proved by Spaltenstein in~\cite{Spaltenstein}.  
Building on work of Tymoczko~\cite{Tymoczko2006}, Ji and the second author construct pavings of Springer fibers in~\cite{Ji-Precup2022} suitable
for use in this paper.  We will use their results, along with the results of Section \ref{sec.affine-subvarieties},
 to construct pavings of extended Springer fibers.

Recall that $\RST(\lambda)$ is the set of all row-strict tableaux of shape $\lambda$.  As in~\cite[Definition 3.1]{Ji-Precup2022},
define the element $\mx_\gl$ of $\co_\lambda$ to be $\mx_\gl = \sum E_{\ell r}$, where the sum is over all pairs $(\ell,r)$ such that 
$r$ labels the box directly to the right of $\ell$ in the base filling of $\lambda$.  
Recall from Section~\ref{sec.Springer.def} that $\spr{\gl}:= \spr{\mx_\gl}=\mu^{-1}(\mx_\gl)$ denotes the Springer fiber of $\mx_{\gl}$.
The following theorem is a combination of results from~\cite{Ji-Precup2022}.

\begin{Thm}\label{thm.paving} Let $\lambda$ be a partition of $n$.  For each $\sigma\in \RST(\gl)$ there exists a  morphism 
\begin{eqnarray}\label{eqn.pavingmap}
f_\sigma: \C^{|\sigma|} \to G
\end{eqnarray}
such that the composition $\C^{|\sigma|} \to G \to G/B = \cb$  is an isomorphism 
onto its image, denoted here by $\cc_\sigma \subseteq \cb_{\gl}$. Moreover, the $\cc_\sigma$ with $\sigma\in \RST(\gl)$ are cells for an affine paving of $\cb_{\gl}$.
\end{Thm}

\begin{Rem} \label{rem:dimension}
The irreducible components of $\spr{\gl}$ are the $\overline{\cc_\sigma}$ where $\sigma$ is a standard tableau.
Thus, if $\sigma$ is standard, then $|\inv(\sigma)| = \sum_i (\lambda_i - 1) = \dim \spr{\gl}$. 
  \end{Rem}

We now adapt the map~\eqref{eqn.pavingmap} to identify $\cc_\sigma$ as a subset of $\spr{\gl}$ in  $G \times^B \fu$.  For each $\sigma\in \RST(\lambda)$, the results of \cite{Ji-Precup2022} imply that $f_\sigma(x)^{-1} \cdot \mx_{\gl} \in \fu$.  This is implicit in the statement of
Theorem~\ref{thm.paving}; it is explained more explicitly in the proof of Proposition
\ref{prop.toric-paving} below.  We
 define $g_{\sigma}: \C^{|\sigma|} \to \fu$ by $g_{\sigma}(x) = f_\sigma(-x)^{-1} \cdot \mx_{\gl}$.
By abuse of notation, we use the same notation
$\cc_\sigma$ as above to denote the image of the map 
$$
F_\sigma: \C^{|\sigma|} \to G \times^B \fu,\; x \mapsto [f_\sigma(x), g_{\sigma}(x) ].
$$
Theorem~\ref{thm.paving} tells us that the $\cc_\sigma$ are cells for an affine paving of $\spr{\gl}$.  

We will need a result, Proposition \ref{prop.toric-paving} below, giving more information about the
map $g_{\sigma}$.  Before stating that proposition, we
introduce some notation.
We view the set of all pairs $(i,j) \in \inv(\sigma)$ as the indexing set
for a basis of $\C^{|\sigma|}$.  Let $\{ x_{(i,j)} \mid (i,j) \in \inv(\sigma) \}$ be coordinates with respect to this basis, and write
$x = (x_{(i,j)})_{(i,j) \in \inv(\sigma)}$ for an element of $\C^{|\sigma|}$.  Let $I = I_{\sigma}$, $J = J_{\sigma}$, and $K = K_{\sigma}$ denote
the subsets of $[n-1]$ defined using $\sigma$ in Section \ref{ss.row-strict}.  We identify $\C^{|I|}$ with the subspace of $\C^{|\sigma|}$ such that
such that all coordinates $x_{(i,j)}$ are zero for $(i,j) \notin \tilde{I}_{\sigma}$.  Similarly, $\C^{|\sigma|-|I|}$ is the subspace defined by the vanishing of the
complementary set of coordinates.  We obtain a decomposition $\C^{|\sigma|} = \C^{|I|} \oplus \C^{|\sigma|-|I|}$
of $\C^{|\sigma|}$ into a direct sum of complementary subspaces.
Let $\pi_I$ denote projection onto the first summand.
 
Recall that $p: \fu \to \toric_{ad}$ is the projection obtained by identifying $\toric_{ad}$ with $\fu/[\fu, \fu]$.   
Let $\cw_{\sigma, ad}$ be the subvariety of $\toric_{ad}$ defined as in Section \ref{ss.aff-mod-finite} by
the equations $x_j = 1$, $x_k = 0$, for $j \in J_{\sigma}$ and $k \in K_{\sigma}$.  The next result describes
the composition $p \circ g_{\sigma}:  \C^{|\sigma|} \to \toric_{ad}$, and shows that the image of this map is~$\cw_{\sigma, ad}$.

\begin{Prop} \label{prop.toric-paving} Let $\sigma\in \RST(\gl)$.  We have $p \circ g_{\sigma} = p  \circ g_{\sigma}  \circ \pi_I$.  Moreover, the
map $p \circ g_{\sigma} |_{\C^{| I |} } \to \toric_{ad}$ is 
given by the formula 
\begin{equation} \label{e.toric-paving}
(x_{(i,\ell(i+1))})_{i \in I_{\sigma}} \mapsto \sum_{i \in I_{\sigma}} x_{(i,\ell(i+1))} E_{i,i+1} + \sum_{i \in J_{\sigma}} E_{i,i+1}.
\end{equation}
Hence $p \circ g_{\sigma} |_{\C^{| I |} } $ is the linear isomorphism~\eqref{eqn.affine.iso} from $\C^{| I |}$ to $\cw_{\sigma, ad}$.
\end{Prop}

\begin{proof} Our proof requires a close inspection of the definition from~\cite{Ji-Precup2022} of the map $f_\sigma$ from \eqref{eqn.pavingmap}.  Let $x = (x_{(i,j)})_{(i,j)\in \inv(\sigma)}\in \C^{|\sigma|}$, and let $\{e_1, e_2 \ldots, e_n\}$ denote the standard basis of $\C^n$. By~\cite[Lemma 4.6, Proposition 5.6]{Ji-Precup2022} the map $f_\sigma$ satisfies the following.
\begin{enumerate}
\item If $r$ does not label a box in the first column of $\sigma$, let $\ell=\ell(r)$ be the label of the box directly to the left of $r$ in $\sigma$.  Then
\[
f_\sigma(x)e_\ell = \mx_\lambda f_{\sigma}(x) e_r + \sum_{\substack{\ell<t\leq n\\ (t,\ell)\in \inv(\sigma)}} x_{(t,\ell)} f_\sigma(x)e_{t}.
\]
\item If $r$ labels a box in the first column of $\sigma$, then $\mx_\lambda f_{\sigma}(x) e_{r} = 0$.
\end{enumerate}
These properties imply that $f_\sigma(x)^{-1} \cdot \mx_{\gl} \in \fu$, as noted above.  This follows since either
$f_\sigma(x)^{-1}  \mx_\lambda f_{\sigma}(x) e_r = 0$, or
\begin{equation} \label{e.in-u}
f_\sigma(x)^{-1}  \mx_\lambda f_{\sigma}(x) e_r = e_{\ell} - \sum_{\substack{\ell<t\leq n\\ (t,\ell)\in \inv(\sigma)}} x_{(t,\ell)} e_{t},
\end{equation}
which is a linear combination of the $e_i$ with $i<r$, as $(t,\ell)\in \inv(\sigma)$ implies $t<r$ by definition of the Springer inversions.

For each $x = (x_{(i,j)})\in \C^{|\sigma|}$ write
\[
p(f_{\sigma}(-x)^{-1}\cdot \mx_\lambda) =  \sum_{i\in [n-1]} a_i(x) E_{i,i+1} \in \fu /[\fu,\fu] = \toric_{ad}.
\]
Here $a_i(x)$ is, by definition, the coefficient of the standard basis vector $e_i$ when we write the vector
\[
f_{\sigma}(-x)^{-1}\mx_\lambda f_\sigma (-x) e_{i+1}
\]
as a linear combination with respect to the standard basis. To prove formula \eqref{e.toric-paving}, we must show that $a_i(x) = x_{(i,\ell(i+1))}$ if $i \in I_{\sigma}$, $a_i(x) = 1$ if $i \in J_{\sigma}$, and
$a_i(x) = 0$ if $i \in K_{\sigma}$.   

If $i+1$ labels a box in the first column of $\sigma$, then $i\in K_\sigma$.  Applying (2) with $r=i+1$ gives us
\[
f_{\sigma}(-x)^{-1}\mathsf{x}_\lambda f_\sigma (-x) e_{i+1} =0
\]
so $a_i(x)=0$ in this case.  If $i+1$ is not in the first column of $\sigma$, let $\ell=\ell(i+1)$ denote the entry directly to its left and apply 
\eqref{e.in-u} with $r=i+1$ to obtain
\[
f_{\sigma}(-x)^{-1}\mx_\lambda f_\sigma (-x) e_{i+1} = e_{\ell} + \sum_{\substack{\ell<t\leq n\\ (t,\ell)\in \inv(\sigma)}} x_{(t,\ell)} e_{t}.
\]
We see that $a_i(x) = 1$ if $\ell=i$, which occurs exactly when $i\in J_\sigma$.  Otherwise, either $(i,\ell)$ is an inversion of $\sigma$ or it is not.  In the former case $i\in I_\sigma$ and $a_i(x) = x_{(i,\ell(i+1))}$, and in the latter, $i\in K_\sigma$ and $a_i(x)=0$.  
This proves \eqref{e.toric-paving}.  The fact that $p \circ g_{\sigma} |_{\C^{| I |} } $ is the linear isomorphism~\eqref{eqn.affine.iso} from $\C^{| I |}$ to $\cw_{\sigma, ad}$ follows directly from the definition of $\cw_{\sigma, ad}$ and map~\eqref{eqn.affine.iso}.
\end{proof}

\begin{example}\label{ex.642-2} Let $n=12$, and $\sigma\in \RST([6 \ 4\ 2])$ be as in Example~\ref{ex.642}.  The computation in the proof of Proposition~\ref{prop.toric-paving} shows that for each $x=(x_{(i,j)})_{(i,j)\in \inv(\sigma)}$,
\begin{eqnarray*}
p(f_{\sigma}(-x)^{-1}\cdot \mx_{[6 \, 4 \, 2]}) &=& E_{1,2}+E_{2,3}+E_{3,4}+ E_{5,6}+E_{6,7}+E_{7,8}+\\
&& \quad \quad\quad\quad\quad\quad\quad\quad\quad E_{9,10}+x_{(10, 4)}E_{10,11}+E_{11,12}.
\end{eqnarray*}
We have $I_\sigma  = \{10\}$ in this case, so $\pi_I(x) = \left(x_{(10,4)}\right)$, and 
\[
p\circ g_{\sigma} (x) =  (1, 1, 1, 0, 1, 1, 1, 0, 1, x_{(10,4)}, 1) = \cw_{\sigma, ad} \subset \toric_{ad}.
\]
\end{example}


\subsection{A paving of the extended Springer fiber} 
This section contains the main result of the paper, which gives an orbifold paving of the extended Springer fiber.  Recall that to each row-strict tableau $\sigma$ of shape $\lambda$, we associated in Section~\ref{ss.row-strict} a set partition $I_\sigma \sqcup J_\sigma\sqcup K_\sigma$ of $[n-1]$.  We write $I=I_\sigma$, $J=J_\sigma$ and $K=K_\sigma$ throughout this section for simplicity.  Notice that $d_*=d_\sigma$, where $d_*$ is defined in the paragraph preceeding the statement of Proposition~\ref{p.components}. The reader may also wish to reacquaint themselves with the notation and discussion preceding the statement of Proposition~\ref{p.zaction2} in Section~\ref{ss.componentsZ}. 

Let $\tilde{ \C}^{|I|}$ be as in Section \ref{ss.componentsZ}.  We define $\tilde{\C}^{|\sigma|} = \tilde{\C}^{|I|} \oplus \C^{|\sigma|-|I|}$.
The map $q: \tilde{ \C}^{|I|} \to \C^{|I|}$ induces a map $\tilde{\C}^{|\sigma|} \to \C^{|\sigma|} $ which is the identity on the factor $\C^{|\sigma|-|I|}$.
By abuse of notation, we denote this map by $q$ as well.  
Recall that $\cw_{\sigma, ad}$ is the subvariety of $\toric_{ad}$ defined as in Section \ref{ss.aff-mod-finite} by
the equations $x_j = 1$, $x_k = 0$, for $j \in J$ and $k \in K$, and let
 $i_{\sigma}:  \C^{|I|} \to   \cw_{\sigma, \ad}$ be the map defined in~\eqref{eqn.affine.iso} of that section.  For each $r\in \Z_{d_\sigma}$ we have a commutative
 diagram,
 \begin{equation} \label{e.commdiag2}
\begin{CD}
 \tilde{\C}^{|I|} @>{i_r}>>   \cw_r  \\
 @VqVV                             @V{\pi}VV\\
 \C^{|I|} @>{i_\sigma}>>    \cw_{\sigma, \ad}
\end{CD}
\end{equation}
where $\cw_r$ is the component of $\pi^{-1}(\cw_{\sigma, ad})$ indexed by $r\in \Z_{d_\sigma}$ and $i_r: \tilde{\C}^{|I|} \to \cw_r$ is the isomorphism of Proposition~\ref{prop.quotient}.

Define $\fu_\sigma:= \{\mathsf{y} \in \fu \mid p(\mathsf{y}) \in \cw_{\sigma, ad}\}$ and $\tilde{\fu}_r = \cw_r \times_{\toric_{\ad}} \fu_\sigma \subset \toric \times_{\toric_{\ad}} \fu = \tilde{\fu}$.
Define $g_{\sigma, r}: \tilde{\C}^{|\sigma|} \to \tilde{\fu}$ by the formula
$$
g_{\sigma, r}(z) = (i_r \circ \pi_I (z), g_{\sigma} \circ q (z) ) \in \toric \times_{\toric_{\ad}} \fu = \tilde{\fu}.
$$
The map $g_{\sigma, r}$ factors through the inclusion $\tilde{\fu}_r \hookrightarrow \tilde{\fu}$ and
the map $g_{\sigma}$ factors through $\fu_{\sigma} \hookrightarrow \fu$.
We obtain a commutative diagram
\begin{equation} \label{e.commdiag3}
\begin{CD}
 \tilde{\C}^{|\sigma |} @>{\bar{g}_{\sigma, r}}>>   \tilde{\fu}_r  @>>> \tilde{\fu} \\
 @VqVV                             @V{\eta_r}VV               @V{\eta}VV \\
 \C^{| \sigma |} @>{\bar{g}_{\sigma}}>>    \fu_{\sigma} @>>>  \fu,
\end{CD}
\end{equation}
where the horizontal maps in the right hand square are the inclusions, 
and $\eta_r = \eta|_{\tilde{\fu}_r}$.  The composition of the top maps
is $g_{\sigma, r}$, and the composition of the bottom maps is~$g_{\sigma}$.  

Finally, define $F_{\sigma,r}:  \tilde{\C}^{|\sigma |} \to \tilde{\cm} = G \times^B \tilde{\fu}$ by the formula
$$
F_{\sigma,r}(z) = [f_{\sigma} \circ q (z), g_{\sigma, r}(z) ]
$$
and consider the commutative diagram
 \begin{equation} \label{e.commdiag4}
\begin{CD}
 \tilde{\C}^{|\sigma |} @>{F_{\sigma,r}}>>   G \times^B \tilde{\fu} = \tilde{\cm} \\
 @VqVV                             @V{\tilde{\eta}}VV\\
 \C^{| \sigma |} @>{F_\sigma}>>   G \times^B \fu = \tilde{\cn}.
\end{CD}
\end{equation}
Let $\tilde{\cc}_{\sigma,r}$ denote the image of $\tilde{\C}^{|\sigma |}$ under $F_{\sigma,r}$.
We can now state our main result.

\begin{Thm} \label{thm.extended-paving} 
 Let $\lambda$ be a partition of $n$, $\sigma\in \RST(\gl)$ and $r \in \Z_{d_\sigma}$.
 \begin{enumerate}
\item The map $F_{\sigma,r}$ takes $\tilde{\C}^{|\sigma|}$ isomorphically onto its image $\tilde{\cc}_{\sigma,r}$, and the preimage of the cell $\cc_\sigma$ under the map $\tilde{\map}: \tilde{\cm} \to \tilde{\cn}$ is the disjoint
union over $r \in \Z_{d_\sigma}$ of the $\tilde{\cc}_{\sigma,r}$.
\item The $\tilde{\cc}_{\sigma,r}$ for $\sigma\in \RST(\lambda)$ and $r \in \Z_{d_\sigma}$ are cells of an orbifold paving of the extended Springer fiber $\espr{\gl} \subseteq \tilde{\cm}$.  
\item The action of  $Z$ on $\espr{\gl}$ satisfies $\underline{\omega} \cdot \tilde{\cc}_{\sigma, i} = \tilde{\cc}_{\sigma, i+1}$.
\end{enumerate}
\end{Thm}

We prove this result by trivializing the bundles involved so that we can more easily
analyze the morphisms.
The quotient map $h: G \to G/B$ is a $B$-principal bundle.   As noted in the previous section,
$h \circ f_{\sigma}: \C^{|\sigma|} \to G/B$ takes $\C^{|\sigma|}$ isomorphically onto its image.
We have a Cartesian square
 \begin{equation} \label{e.commdiag5}
\begin{CD}
 \C^{| \sigma |} \times B @>>>   G  \\
 @VVV                             @VhVV\\
 \C^{| \sigma |} @>{h \circ f_{\sigma}}>>   G/B,
\end{CD}
\end{equation}
where the top map takes $(x,b)$ to $f_{\sigma}(x) b$.  This trivializes the pullback of the bundle
$G \to G/B$ along the map $h \circ f_{\sigma}$.  This trivialization induces a trivialization of the bundle
$\tilde{\cn} = G \times^B \fu \to G/B$:
 \begin{equation} \label{e.commdiag6}
\begin{CD}
 \C^{| \sigma |} \times \fu @>{\psi}>>   G \times^B \fu = \tilde{\cn}  \\
 @VVV                             @VVV\\
 \C^{| \sigma |} @>{h \circ f_{\sigma}}>>   G/B,
\end{CD}
\end{equation}
where $\psi(x,\mathsf{y}) = [f_{\sigma}(x), \mathsf{y}]$.  Similarly, we obtain a trivialization of the bundle $ \tilde{\cm}=G\times^B \tilde{\fu} \to G/B$:
 \begin{equation} \label{e.commdiag7}
\begin{CD}
 \C^{| \sigma |} \times \tilde{\fu} @>{\tilde{\psi}}>>   G \times^B \tilde{\fu} = \tilde{\cm}  \\
 @VVV                             @VVV\\
 \C^{| \sigma |} @>{h \circ f_{\sigma}}>>   G/B,
\end{CD}
\end{equation}
where $\tilde{\psi}(x,y) = [f_{\sigma}(x),y]$. 

We expand the commutative diagram \eqref{e.commdiag4} by factoring the maps $F_{\sigma,r}$ and $F_{\sigma}$:
 \begin{equation} \label{e.commdiag8}
\begin{CD}
 \tilde{\C}^{|\sigma |} @>{q \times \bar{g}_{\sigma,r}}>>     \C^{| \sigma |} \times \tilde{\fu}_r  @>>> \C^{| \sigma |} \times \tilde{\fu} @>{\tilde{\psi}}>>   G \times^B \tilde{\fu} = \tilde{\cm} \\
 @VqVV                     @V{\small{\mbox{id}} \times \eta_r}VV       @V{\small{\mbox{id}} \times \eta}VV   @V{\tilde{\eta}}VV\\
 \C^{| \sigma |}        @>{\small{\mbox{id}} \times \bar{g}_{\sigma}}>>   \C^{| \sigma |} \times \fu_\sigma  @>>>   \C^{| \sigma |} \times \fu  @>{\psi}>>   G \times^B \fu = \tilde{\cn}.
 \end{CD}
\end{equation}
Note  that the morphism $\mbox{id} \times g_{\sigma}: \C^{| \sigma |}  \to \C^{| \sigma |} \times \fu $ is the graph morphism
of $g_{\sigma}$.  The morphism $q \times g_{\sigma,r}: \tilde{\C}^{|\sigma |} \to \C^{| \sigma |} \times \tilde{\fu} $ is the composition
of the graph morphism of $g_{\sigma,r}$ with the projection $q \times \mbox{id}: \tilde{\C}^{|\sigma |} \times \tilde{\fu}
\to \C^{| \sigma |} \times \tilde{\fu}$.

Define $\tilde{\cc}'_{\sigma,r}$ and $\cc'_{\sigma}$ to be the images of $ \tilde{\C}^{|\sigma |} $ and $ \C^{| \sigma |} $
under the finite morphisms $q \times \bar{g}_{\sigma,r}$ and $\small{\mbox{id}} \times \bar{g}_{\sigma}$, respectively.  
Because the middle horizontal maps are inclusions of closed schemes, we can view $\tilde{\cc}'_{\sigma,r}$
and $\cc'_{\sigma}$ (respectively) as closed subschemes of $ \C^{| \sigma |} \times \tilde{\fu} $ and $\C^{| \sigma |} \times \fu$.

\begin{proof}[Proof of Theorem~\ref{thm.extended-paving}] (1) Since $\tilde{\psi}$ is an isomorphism
onto its image, to show that $q \times g_{\sigma,r}$ takes $ \tilde{\C}^{|\sigma |}$ isomorphically onto $\tilde{\cc}_{\sigma,r}$,
 it suffices
to prove that $q \times \bar{g}_{\sigma,r}$ takes $ \tilde{\C}^{|\sigma |}$ isomorphically onto $\tilde{\cc}'_{\sigma,r}$.
Define a morphism $\nu:  \C^{| \sigma |} \times \tilde{\fu}_r \to  \tilde{\C}^{|\sigma |}$ as follows.
Recalling that $\C^{| \sigma |} \times \tilde{\fu}_r = (\C^{|I|} \times \C^{|\sigma|-|I|}) \times (\cw_r\times_{\toric_{ad}} \fu_\sigma)$, let $\zeta=(x_1,x_2,v, \mathsf{y})$ be an element of $\C^{|\sigma|}\times \fu_r$.  Recall also that $\tilde{\C}^{|\sigma |}  =  \tilde{\C}^{|I|}  \times \C^{|\sigma| - |I|}$ and define $\nu(\zeta)=(j_r(v), x_2)$, where $j_r: \cw_r \to \tilde{\C}^{|I|}$ is the inverse of the isomorphism $i_r$.  We have
$$
\tilde{\C}^{|\sigma |} \xrightarrow{\;q \times \bar{g}_{\sigma,r}\;} \C^{| \sigma |} \times \tilde{\fu}_r \stackrel{\nu}{\longrightarrow}
\tilde{\C}^{|\sigma |}.
$$
Using the definitions of the maps involved, given $z=(z_1, z_2)\in \tilde{\C}^{|\sigma |} $ we have
\[
(\nu  \circ (q \times \bar{g}_{\sigma,r}))(z) = \nu (q(z_1), z_2, i_r(z_1), g_\sigma \circ q(z)) = ((j_r\circ i_r)(z_1), z_2)=z.
\]
Thus $ \nu  \circ (q \times \bar{g}_{\sigma,r})$ is the identity, so $q \times \bar{g}_{\sigma,r}$ takes
$\tilde{\C}^{|\sigma |}$ isomorphically onto its image $\tilde{\cc}'_{\sigma,r}$, as desired.

To show that $\tilde{\eta}^{-1}(\cc_{\sigma})$ is the disjoint union of the $\tilde{\cc}_{\sigma,r}$,
it suffices to show that $(\mbox{id} \times \eta)^{-1}(\cc'_{\sigma})$ is the disjoint union of the $\tilde{\cc}'_{\sigma,r}$.
The $\cw_r$ are disjoint for different values of $r$ (by 
Lemma \ref{lem.disjoint} and Proposition \ref{p.components}), hence so are the $\C^{| \sigma |} \times \tilde{\fu}_r$,
and hence the $\tilde{\cc}'_{\sigma,r}$ are as well, since $\tilde{\cc}'_{\sigma,r}$ is contained in $\C^{| \sigma |} \times \tilde{\fu}_r$.  The commutativity of \eqref{e.commdiag8} implies 
\begin{eqnarray*}
(\mbox{id}\times \eta)(\tilde{\cc}'_{\sigma, r}) &=& \left((\mbox{id}\times \eta) \circ (q\times g_{\sigma, r})\right)(\tilde{\C}^{|\sigma|})\\
&=& \left( (\mbox{id}\times g_\sigma)\circ q\right)(\tilde{\C}^{|\sigma|})
=  (\mbox{id}\times g_\sigma)(\C^{\sigma}) = \cc_\sigma'.
\end{eqnarray*}
Therefore $\bigcup_r \tilde{\cc}'_{\sigma,r} \subseteq (\mbox{id} \times \eta)^{-1}(\cc'_{\sigma})$.

We now prove the reverse inclusion.  Let $\zeta=(x_1, x_2, v, \mathsf{y}) \in \C^{|\sigma|} \times \tilde{\fu}$ such that $(\mbox{id}\times \eta)(\zeta)\in \cc_\sigma'$.  By definition of the map $\eta$ and scheme $\cc_\sigma'$ and by Proposition~\ref{prop.toric-paving}, our assumptions imply that $\mathsf{y} = g_\sigma (x_1, x_2)$ and $\pi(v) = p \circ g_\sigma(x_1) \in \cw_{\sigma, ad}$.  In particular, we see that $v\in \cw_r$, and therefore $\zeta\in \C^{| \sigma |} \times \tilde{\fu}_r$ for some $r\in \Z_{d_\sigma}$.   Let $\tilde{\zeta} = \nu(\zeta) = (j_r(v), x_2) \in  \tilde{\C}^{| I |} \times \C^{|I| - |\sigma|} = \tilde{\C}^{|\sigma |} $.
We claim that
 $\zeta = (q \times \bar{g}_{\sigma,r})(\tilde{\zeta})$.  This suffices, since then $\zeta \in \tilde{\cc}'_{\sigma,r}$, and the reverse inclusion follows.  To prove the claim, we begin with the observation that $\pi(v) = i_\sigma (x_1) $ by Proposition~\ref{prop.toric-paving}. Since the diagram~\ref{e.commdiag2} commutes and $i_\sigma$ is an isomorphism, it follows that 
\[
(i_\sigma\circ q )(j_r(v)) = (\pi \circ i_r)(j_r(v)) = i_\sigma(x_1) \Rightarrow q(j_r(v)) = x_1.
\] 
Thus, 
\[
(q\times \bar{g}_{\sigma, r}) (\tilde{\zeta}) = (q(j_r(v)), x_2, i_r\circ j_r(v), g_\sigma(q(j_r(v)), x_2)) = (x_1, x_2, v, g_\sigma (x_1, x_2)) = \zeta,
\]
as desired.

(2) The $\cc_{\sigma}$ for $\sigma \in \RST(\lambda)$ form the cells of an affine paving of $\spr{\lambda}$.  
This means that $\spr{\lambda}$ is the disjoint union of the $\cc_{\sigma}$, and moreover, that
there is a filtration of $\spr{\lambda}$ by closed subspaces
$$\spr{\lambda,0} \subset \spr{\lambda,1} \subset \cdots \subset \spr{\lambda,k} = \spr{\lambda}$$  such that each
$ \spr{\lambda,i} -  \spr{\lambda,i-1}$ is a disjoint union of the
cells $\cc_{\sigma}$ it contains.
The inverse images $\espr{\lambda,i} =  \tilde{\eta}^{-1}( \spr{\lambda,i} )$ form a filtration of $\espr{\lambda}$ by
closed subspaces such that each $ \espr{\lambda,i} -  \espr{\lambda,i-1}$ is a disjoint union
of the $\tilde{\eta}^{-1}(\cc_{\sigma})$ it contains.  By (1),  $\tilde{\eta}^{-1}(\cc_{\sigma})$ is a disjoint union of
$\tilde{\cc}_{\sigma,r}$.  Hence the $\tilde{\cc}_{\sigma,r}$ are cells of an orbifold paving of $\espr{\gl} \subseteq \tilde{\cm}$.

(3) Proposition \ref{p.zaction2} implies that $\underline{\omega} \circ g_{\sigma,r} = g_{\sigma,r+1}$, which in turn implies
that $\underline{\omega} \circ F_{\sigma,r} = F_{\sigma,r+1}$.  Since $\cc_{\sigma,r}$ is the image of $F_{\sigma,r}$, the
result follows.
\end{proof}


\section{Poincar\'e polynomials} \label{sec.combinatorics}

\subsection{Divisible tableaux}

Suppose that $\lambda = [\lambda_1 \ \cdots \ \lambda_\ell]$ is a partition of $n$ with $\ell$ parts.  If $d$ divides $\lambda_i$ for each $1\leq i\leq \ell$, we define a partition of $\frac{n}{d}$ by $\lambda/d := [\lambda_1/d \ \cdots \ \lambda_\ell/d]$.  Suppose that $\sigma\in \RST(\lambda)$ 
and that $d$ is a divisor of $\sigma$.  We form a \textit{quotient tableau} of shape $\lambda/d$, denoted $\sigma/d$, by merging consecutive blocks in the rows of $\sigma$ and labeling the merged blocks in the same relative order in which they appeared in $\sigma$.  Thus, the block containing $md$ in $\sigma$ will be labeled with $m$ in $\sigma/d$ for each $1\leq m\leq n/d$. Since $\sigma$ is row-strict, $\sigma/d$ is as well.  

\begin{example}\label{ex.divtab} The row strict tableau $\sigma$ of shape $\lambda=[4^2 \ 2^2]$ from Example~\ref{ex.IJK} is displayed below, together with the corresponding quotient tableau $\sigma/2$ of shape $\lambda/2 = [2^2 \ 1^2]$.
\[\ytableausetup{centertableaux} \begin{ytableau}3 & 4 & 5 & 6\\ 1 & 2 & 9 & 10\\ 7 & 8\\ 11 & 12\end{ytableau} \quad\quad\quad\quad\quad \begin{ytableau} 2 & 3\\ 1 & 5\\ 4 \\ 6 \end{ytableau}
\]
\end{example}

Recall that $|\sigma|$ denotes the number of Springer inversions of $\sigma$ (Definition~\ref{def.Springer.inv}).  The main result of this section proves a precise relationship between $|\sigma |$ and $|\sigma/d|$.  First, we need a relaxation of our Springer inversion definition. 

\begin{Def} \label{def.Springer.pair} Let $i,j\in [n]$ such that $i>j$.  We say that the pair $(i,j)$ is a \textit{Springer pair} of $\sigma\in \RST(\lambda)$ if 
\begin{enumerate}
\item $i$ occurs in the same column as $j$ or in any column strictly to the left of the column containing $j$ in $\sigma$, and
\item if the box directly to the right of $j$ in $\sigma$ is labeled by $r$, then $i< r$.
\end{enumerate}
We will denote the set of Springer pairs for $\sigma$ by $\pairs(\sigma)$ and the size of this set by $|\pairs(\sigma)|$.  
\end{Def}
Comparing this with Definition~\ref{def.Springer.inv} (the definition of Springer inversions), we see that for a given $\sigma\in \RST(\lambda)$,
all Springer inversions are Springer pairs, but there can be more Springer pairs than Springer inversions. Indeed, condition (2) of Definitions~\ref{def.Springer.pair} and~\ref{def.Springer.inv} are the same, but condition (1) of Definition~\ref{def.Springer.pair} includes all pairs $i>j$ satisfying (2) such that $j$ appears below $i$ and in the same column while condition (1) of Definition~\ref{def.Springer.inv} excludes these pairs.    Note that if $\sigma$ is a standard tableau, then the Springer pairs and Springer inversions coincide.

\begin{example} The pair $(4,2)$ is a Springer pair for the row-strict tableau of shape $[4^2 \ 2^2]$ appearing in Examples~\ref{ex.IJK} and~\ref{ex.divtab}, but not a Springer inversion (see Example~\ref{ex.inversions}).
\end{example}

\begin{Prop}\label{pairs.invers}
Let $\sigma \in \RST(\lambda)$.  If $d$ is a divisor of $\sigma$, then \begin{equation}|\pairs(\sigma)| - |\sigma | = |\pairs(\sigma/d)| - |\sigma/d|.\end{equation}
\end{Prop}

\begin{proof} Without loss of generality, we may assume $d>1$.  Suppose $\ell>k$.  As the entries in each block are consecutive, 
$(\ell,k)$ can satisfy condition (2) of the Springer pair or Springer inversion definition only if $k$ labels the cell at the end of a block. Furthermore, if $(\ell,k)$ is a Springer pair but not an inversion, then both $\ell$ and $k$ must occur at the end of a block (since $\ell$ and $k$ are in the same column).  This shows that every Springer pair of $\sigma$ that is not an inversion is of the form $(\ell, k)=(di, dj)$, where $(i,j)$ is a Springer pair of $\sigma/d$ that is not a Springer inversion.

Conversely, suppose that $(i,j) \in \pairs(\sigma/d)$, but $(i,j)$ is not a Springer inversion of $\sigma/d$.  Then $i$ is in the same column as $j$ but above $j$.  This means that in $\sigma$, the boxes containing values $(i-1)d+1,(i-1)d+2,\dots,id$ appear in a row above the boxes containing $(j-1)d+1,(j-1)d+2,\dots,jd$.  Let $r$ denote the entry labeling the box to the right of $j$ in $\sigma/d$, if it exists.  We have the following configuration of blocks in $\sigma$.
\[
\ytableausetup{centertableaux, boxsize=0.56in} \begin{ytableau} {\scriptstyle(i-1)d+1} & {\scriptstyle(i-1)d+2} & \cdots & {\scriptstyle id} \\ \none[\vdots] & \none[\vdots]  & \none[\vdots]  & \none[\vdots]  \\   {\scriptstyle(j-1)d+1} &  {\scriptstyle(j-1)d+2} & \cdots &  {\scriptstyle jd} &  {\scriptstyle(r-1)d+1}  & {\scriptstyle(r-1)d+2} & \cdots & {\scriptstyle rd} \end{ytableau}
\]
As $(i,j)$ is a Springer pair for $\sigma/d$, we have
\[
i< r \Rightarrow i \leq (r-1) \Rightarrow id \leq (r-1)d \Rightarrow id<(r-1)d+1.
\] 
We see that $(id,jd)$ is a Springer pair for
$\sigma$ that is not a Springer inversion.  Note that this is also true in the case where $j$ labels a box at the end of a row, i.e., the case in which $r$ does not exist. 

The preceding discussion shows that there is a bijection $(i,j) \mapsto (di,dj)$ between Springer pairs of $\sigma/d$ 
which are not Springer inversions
and Springer pairs of $\sigma$ which are not Springer inversions.
The result follows.
\end{proof}

We rearrange the equality in Proposition~\ref{pairs.invers} to obtain
\begin{Cor}\label{cor.pairs}
If $\sigma \in \RST(\lambda)$ and $d$ is a divisor of $\sigma$, then \begin{equation}|\sigma | = |\pairs(\sigma) | -|\pairs(\sigma/d)| +|\sigma /d|.\end{equation}
\end{Cor}

If $\sigma$ is row strict, we can rearrange the entries in each column of $\sigma$ into decreasing order to obtain a standard tableau.  This is called the \emph{standardization of $\sigma$}, denoted here by $\std(\sigma)$. We use this operation in the proof of the following result.

\begin{Prop}\label{dim.pairs}
For all $\sigma \in \RST(\lambda)$, $|\pairs(\sigma)| = \dim \spr{\gl}$.
\end{Prop}

\begin{proof}
We claim that $|\pairs(\sigma)|$ is invariant under the standardization operation, so $|\pairs(\sigma)| = |\pairs(\std(\sigma))|$.  
This suffices, since because $\std(\sigma)$ is standard, all of its Springer pairs are also Springer inversions.  Thus, Remark \ref{rem:dimension} implies
that $|\pairs(\std(\sigma))|=|\std(\sigma) | = \dim \spr{\gl}$.

To prove the claim, we show that for $i \in [n-1]$, the the number of pairs of the form $(i, j)$ of $\sigma$ depends only on the entries in the column containing $i$, so it is unchanged by standardization.  
Define $a_m(i)$ to be the number of entries in the $m$-th column of $\sigma$ that are less than $i$.   Suppose $i$ is in the $k$-th column.  We will prove 
that the number of pairs of the form $(i,j)$ is $a_k(i)$.

To count pairs $(i,j)$ of $\sigma$, by condition (1) of Definition~\ref{def.Springer.pair}, we need not consider any $j<i$ such that $j$ is in a column to the left of $i$.  
Suppose $\sigma$ has $t$ total columns, and consider the nonnegative integers $a_k(i),\dots, a_t(i)$.  Since $\sigma$ is row strict, this sequence must be decreasing, so that $a_k(i)\geq a_{k+1}(i)\geq\cdots\geq a_t(i)$. For each $m$ such that $k\leq m <t$, each of the $a_{m+1}(i)$ entries in the $m+1$-st column must be preceded by an entry in the $m$-th column that is also less than $i$.  This leaves precisely $a_{m}(i) - a_{m+1}(i)$ entries in the $m$-th column to satisfy condition (2) of Definition~\ref{def.Springer.pair}.  For the final column, there is no need to check this condition.  Thus $i$ appears in exactly
\begin{equation}
\left(\sum_{m=k}^{t-1} a_m(i) - a_{m+1}(i)\right) + a_t(i) = a_k(i)
\end{equation}
Springer pairs of the form $(i,j)$, proving the claim.
\end{proof}

Combining Proposition \ref{dim.pairs} with Corollary \ref{cor.pairs} yields a simple relation between $|\sigma |$ and $|\sigma/d |$.
Recall that $\spr{\gl/d}$ is a Springer fiber for the group $SL_{n/d}(\C)$ over a nilpotent element of Jordan type $\gl/d$. (see Section \ref{sec.Springer.def})

\begin{Cor} \label{cor.shift}
If $\sigma \in \RST(\lambda)$ and $d$ is a divisor of $\sigma$, then \begin{equation}|\sigma | =  \dim \spr{\gl}-\dim \spr{\gl/d} +|\sigma /d|.\end{equation}
\end{Cor}

If $d$ is a common divisor of all the parts of $\lambda$, denoted below by $d|\gl$, we define 
\begin{equation} \label{e.dgld}
D_{\gl, d} = \dim  \spr{\gl} -\dim \spr{\gl/d}.
\end{equation}
Equivalently, if $\mu$ is a partition of $n/d$ such that $\mu = \gl/d$ we write $\mu| \gl$ and say that $\mu$ divides $\lambda$.  In that case, we write 
$$
D_{\gl, \mu} = \dim \spr{\gl}-\dim \spr{\mu}.
$$
Since $\dim \spr{\gl} = \sum_i (i-1)\gl_i$, we see that $\dim\spr{\gl/d} = \frac{1}{d} \dim \spr{\gl}$.  
Hence
\begin{equation} \label{e.dim.diff}
D_{\gl, d} = \frac{d-1}{d} \sum_i (i-1)\gl_{i} .
\end{equation}

\begin{example} Let $\sigma$ and $\sigma/2$ be the row strict tableaux of shape $\lambda = [4^2 \ 2^2]$ and $\lambda/2 = [2^2 \ 1^2]$, respectively, appearing in Example~\ref{ex.divtab} above.  The Springer inversions of $\sigma$ were computed in Example~\ref{ex.inversions} and we also have 
\[
\inv(\sigma/2) = \{ (6,4), (6,5), (6,3), (5,3), (4,1), (4,3) \}.
\]
The reader can confirm that the map $(i,j) \mapsto (2i, 2j)$ defines an injection from $\inv(\sigma/2)$ to $\inv(\sigma)$, as established in the proof of Proposition~\ref{pairs.invers} above. Thus $|\sigma|-|\sigma/2| = 13-6= 7$.  This confirms the results of Corollary~\ref{cor.shift}, as $\dim \spr{[4^2 \ 2^2]} - \dim \spr{[2^2 \ 1^2]} = 14-7 = 7$ in this case.
\end{example}


\subsection{Poincar\'e polynomials of extended Springer fibers}
Let $\mathbb{F}$ denote a field of characteristic zero or of characteristic prime to $n$.
We take homology and cohomology with coefficients in $\mathbb{F}$, and write simply
$H_i(\cy)$ and $H^i(\cy)$ for $H_i(\cy;\mathbb{F})$ and $H^i(\cy;\mathbb{F})$.  If $\cy$ is a variety whose odd-dimensional
cohomology vanishes, we define its \textit{modified Poincar\'e polynomial} by
$$
P(\cy;t) := \sum_j \dim H^{2j}(\cy) \ t^j,
$$  
so $P(\cy; t^2)$ is the usual Poincar\'e polynomial of $\cy$.  By abuse of terminology, we will simply
refer to $P(\cy;t)$ as the Poincar\'e polynomial of $\cy$.  If $Z \cong \Z_n$ acts on $\cy$, then it acts on each $H^{2j}(\cy)$.
For each character $\chi_i$ of $Z$, let $H^{2j}(\cy)_{\chi_i}$ denote the $\chi_i$-isotypic component
of $H^{2j}(\cy)$, and define the \textit{$\chi_i$-isotopic Poincar\'e polynomial} by the formula
$$
P_{\chi_i}(\cy; t) := \sum_j \dim H^{2j}(\cy)_{\chi_i} \ t^j.
$$
It is convenient to combine all the isotypic Poincar\'e polynomials into the \textit{$Z$-equivariant Poincar\'e polynomial}
defined by the formula
$$
\tilde{P}(\cy; t) = \sum P_{\chi_i}(\cy;t)  \chi_i \in \widehat{Z}\otimes\Z[t].
$$
$\tilde{P}(\cy;t)$ may be thought of as a map from $Z$ to $\Z[t]$; evaluating at the identity element
of $Z$ yields $P(\cy;t) = \sum_i  P_{\chi_i}(\cy;t)$, recovering the Poincar\'e polynomial as a sum of
isotypic Poincar\'e polynomials.  

The main result of this section, Theorem \ref{thm.poincare}, shows
that up to a shift, the isotypic Poincar\'e polynomials $P_{\chi_i}(\espr{\gl}; t)$ of the extended Springer fibers 
$\espr{\gl}$ equal the Poincar\'e polynomials of ordinary Springer fibers
for smaller rank groups.
This yields a description of the 
cohomology stalks of the
corresponding Lusztig sheaves considered by the authors in~\cite{GPR}; see Theorem~\ref{thm.lusztigsheaf}.

Let $V_i$ denote the $1$-dimensional representation of $Z$ with character $\chi_i$.  
For each divisor $d$ of $n$, let $E_d$ be the $d$-dimensional representation of $Z$ where $Z$ acts by cyclically permuting a basis. 
Then as representations of $Z$,
\begin{equation} \label{e.decomp}
E_d \cong V_0 \oplus V_{n/d} \oplus V_{2n/d} \oplus \cdots \oplus V_{(d-1)n/d}.
\end{equation}

\begin{Prop}  \label{prop.cohomologyZaction} 
For $i \geq 0$, $H^{2i+1}(\espr{\gl}) = 0$.
As a representation of $Z$,
$$
H^{2i}(\espr{\gl}) = \bigoplus_{\sigma\in \RST(\gl), |\sigma|=i} E_{d_{\sigma}}
$$
and
$$
P(\espr{\gl};t) = \sum_{\sigma\in \RST(\gl)} d_{\sigma} t^{ |\sigma|}.
$$
\end{Prop}

\begin{proof}
This is an immediate consequence of Theorem \ref{thm.extended-paving}.
\end{proof}

We say a tableau is \textit{indivisible} if it is not divisible by any $d>1$.  Let
$\IRST(\gl)$ be the set of indivisible row-strict tableaux of shape $\gl$, and define 
$$
Q_\gl(t) := \sum_{\sigma \in \IRST(\gl)} t^{|\sigma|}.
$$
Using the affine paving of the Springer fiber $\spr{\gl}$, we see that
\begin{align*}
P( \spr{\gl};t) &= \sum_{\sigma \in \RST(\gl)}  t^{|\sigma|} = \sum_{d|\gl} \, \sum_{\substack{\sigma\in \RST(\gl)\\d_\sigma=d}} t^{|\sigma|} \\ &= \sum_{d|\gl} t^{D_{\gl,d}} \sum_{\substack{\sigma\in \RST(\gl)\\d_\sigma=d}} t^{|\sigma/d|}
= \sum_{d | \gl} t^{D_{\gl,d}}Q_{\gl/d}(t) , 
\end{align*}
where the third equality follows from Corollary \ref{cor.shift}, and the last from the fact that $d=d_\sigma$ and the definition of $Q_{\lambda/d}(t)$.  We can rewrite this as a sum over partitions
dividing $\gl$:
\begin{align}\label{eqn.poincare.Sp}
P(\spr{\gl};t) =  
\sum_{\mu | \gl} t^{D_{\gl,\mu}} Q_\mu(t).
\end{align}
We are now ready to prove the main result of this section, which describes the isotypic Poincar\'e polynomial of a generalized Springer fiber.

\begin{Thm} \label{thm.poincare}
Let $i \in \{0, \ldots, n-1 \}$, and set $d = n/\gcd(n,i)$.  Then $P_{\chi_i}(\espr{\gl};t)=0$ unless $d|\gl$, in which case
$$
P_{\chi_i}(\espr{\gl};t ) =  t^{D_{\gl,d}} P(\spr{\gl/d};t). 
$$
\end{Thm}

\begin{proof}
By Proposition \ref{prop.cohomologyZaction}  and equation~\eqref{e.decomp},
$$
\tilde{P}(\espr{\gl};t ) = \sum_{a | \gl} \sum_{\substack{\sigma \in \RST(\gl)\\ d_{\sigma} = a}} (\chi_0 + \chi_{n/a} + \cdots + \chi_{(a-1)n/a}) t^{|\sigma|}.
$$
In light of Corollary \ref{cor.shift}, this can be rewritten as
$$
\tilde{P}(\espr{\gl};t ) = \sum_{a | \gl} t^{D_{\gl,a}} \sum_{\sigma \in \IRST(\gl/a)}  (\chi_0 + \chi_{n/a} + \cdots + \chi_{(a-1)n/a}) t^{|\sigma|}.
$$
The isotypic Poincare polynomial $P_{\chi_i}( \espr{\gl};t )$ is the coefficient of $\chi_i$ on the right hand side of the above.  We see
that $\chi_i$ occurs in the summands corresponding to $a$ such that $\frac{n}{a}$ divides $i$, or equivalently, by Lemma
\ref{lem.elementary} below, such that  $d=\frac{n}{\gcd(n,i)}$ divides $a$.   Thus,
\begin{eqnarray}\label{eqn.isotypic.pf}
P_{\chi_i}(\espr{\gl};t ) = \sum_{\substack{a \textup{ such that}\\ a|\lambda\textup{ and }d|a}} t^{D_{\gl,a}}  \sum_{\sigma \in \IRST(\gl/a)}  t^{|\sigma|}.
\end{eqnarray}
Let $\mu= \gl/a$.  By definition, 
$D_{\gl,a} = D_{\lambda, \mu}$; the condition that $d$ divides $a$ is equivalent to $\mu \vert (\gl/d)$.  
Reindexing the sum in~\eqref{eqn.isotypic.pf} over all tableaux dividing $\lambda/d$, we obtain
$$
P_{\chi_i}(\espr{\gl};t )  = \sum_{\mu | (\gl/d)} t^{D_{\gl, \mu}} \sum_{\sigma \in \IRST(\mu)} t^{|\sigma|} = 
\sum_{\mu | (\gl/d)} t^{D_{\gl, \mu}} Q_\mu(t) .
$$
Since
\[
D_{\lambda, \mu} = D_{\gl, (\gl/d)} + D_{(\gl /d), \mu} = D_{\gl,d} + D_{(\gl /d), \mu},
\]
we have
$$
P_{\chi_i}(\espr{\gl};t ) = t^{D_{\gl, d}} \sum_{\mu | (\gl/d)} t^{D_{(\gl/d), \mu}} Q_\mu(t) = t^{D_{\gl, d}}  P( \spr{\gl/d};t)
$$
where the last equality follows from~\eqref{eqn.poincare.Sp}. This completes the proof.
\end{proof}

In the previous proof, we used the following elementary lemma.  Our convention is that $\gcd(n,0) = n$.

\begin{Lem} \label{lem.elementary}
Suppose $a$ divides $n$, and let $i$ be a nonnegative integer.  Then 
$\frac{n}{a}$ divides $i$ if and only if  $\frac{n}{\gcd(n,i)}$ divides $a$.
\end{Lem}

\begin{proof}
If $i = 0$, then the assertions $\frac{n}{a}$ divides $i$ and $\frac{n}{\gcd(n,i)}$ divides $a$ are both
true, so we may assume $i>0$.

$(\Rightarrow)$ By hypothesis, $i = \frac{cn}{a}$ for some positive integer $c$, so $ai = cn$.
Hence $\frac{ai}{\gcd(n,i)} = \frac{cn}{\gcd(n,i)}$.
Since $\frac{i}{\gcd(n,i)}$ and $\frac{n}{\gcd(n,i)}$ are relatively prime, we
see that $\frac{n}{\gcd(n,i)}$ divides $a$.

$(\Leftarrow)$ By hypothesis, $a = \frac{cn}{\gcd(n,i)}$ for some positive integer $c$, 
so $\gcd(n,i) = \frac{cn}{a}$.  Hence $\frac{n}{a}$ divides $\gcd(n,i)$; since $\gcd(n,i)$ divides $i$, the result follows.
\end{proof}

Theorem~\ref{thm.poincare} yields another formula for the Poincar\'e polynomial of the extended Springer fiber in terms of those of smaller rank Springer fibers.

\begin{Cor}\label{cor.poincare} Let $\lambda$ be a partition of $n$.  Then 
\[
P(\espr{\gl}; t) = \sum_{d|\lambda} \varphi(d)  t^{D_{\gl, d}} P( \spr{\gl /d}; t)
\]
where $ \varphi(d)$ denotes Euler's totient function.
\end{Cor}

\begin{proof} Since the Poincar\'e polynomial of $\espr{\gl}$ is a sum of the isotypic Poincar\'e polynomials, the formula follows immediately once we establish that 
$$\varphi(d)= \left|\left\{ i \mid 0\leq i \leq n-1,\, d = \frac{n}{\gcd(n,i) } \right\}\right|.$$  
If $d = 1$, then $\varphi(d) = 1 = | \{ 0 \} |$, so above equation holds.  So assume $d>1$.
The condition $d = \frac{n}{\gcd(n,i)}$ is equivalent to $\gcd(n,i) = \frac{n}{d}$, and if this holds, then $i = \frac{kn}{d}$ for some $1\leq k\leq d$.  The condition
$\gcd(n,\frac{kn}{d}) = \frac{n}{d}$ holds if and only if $\gcd(d,k) = 1$.  By definition of the totient function, there are precisely $\varphi(d)$ possibilities for $k$, and
so there are precisely $\varphi(d)$ many $i$ with the desired property.
\end{proof}

\begin{example} Let $n=12$ and set $\lambda=[6^2]$.  By Theorem~\ref{thm.poincare}, $P_{\chi_i}(\espr{\gl};t)=0$ unless $i\in \{ 0,2,4,6,8,10 \}$, as these are the values of $i$ such that $d=\frac{12}{\gcd(12,i)}$ divides $\lambda$.  Let $i=4$, so $d=3$ and $D_{\gl,3} = 6-2=4$.
By the theorem,
\[
P_{\chi_4}(\espr{[6^2]} ;t) = t^4 P(\spr{[2^2]}; t).
\]
Corollary~\ref{cor.poincare} gives us
\begin{align*}
P(\espr{[6^2]};t) = P(\spr{[6^2]};t)+ t^3 P(\spr{[3^2]} ;t)
+ 2t^4 P(\spr{[2^2]};t) + 2t^5 P(\spr{[1^2]};t).
\end{align*}
\end{example}


\subsection{The Poincar\'e polynomials of Lusztig sheaves}
In this section, we apply our results to the Lusztig sheaves, which play an important
role in Lusztig's generalized Springer correspondence.  In particular, we prove that
in type $A$, each Lusztig sheaf bears a close resemblance to the Springer sheaf for
a smaller rank group (see Corollary \ref{cor:smallergroup} for a precise statement).

We begin by recalling some results about the Lusztig sheaves.
For each central character $\chi \in \widehat{Z}$, there is a \textit{Lusztig sheaf} $\A_{\chi}$,
which is an element in the constructible derived
category of $\cn$ on which the center $Z$ acts by the character $\chi$.  (See~\cite{GPR} or~\cite[\S8.4-8.5]{Achar-book} for a precise definition, and~\cite[\S6.2, Theorem 8.5.8]{Achar-book} for the central character.)  The Lusztig sheaf corresponding to the trivial character is the \textit{Springer sheaf} $ \A_{\mathrm{id}} := \mu_*\shq_{\tilde{\cn}}[\dim \cn]$. 

The main result of \cite{GPR} connects the Lusztig sheaves to the generalized Springer
resolution.  The pushforward
$\psi_* \shq_{\tilde{\cm}}[\dim\cn]$ is an element in the constructible derived 
category of $\cn$, and  by~\cite[Theorem 5.3]{GPR},
\begin{eqnarray}\label{eqn.pushforward}
\psi_* \shq_{\tilde{\cm}}[\dim\cn] = \bigoplus_{\chi \in \widehat{Z}} \A_{\chi}.
\end{eqnarray}
Taking stalks in~\eqref{eqn.pushforward} and using proper base change, we see that the cohomology of the extended Springer fibers and 
the stalks of the cohomology sheaves of the complex $\psi_* \shq_{\tilde{\cm}}[\dim\cn]$ are related by
\begin{eqnarray}\label{eqn.stalks}
H^{2j+N}(\espr{\mx_\gl}) \cong \bigoplus_{\chi \in \widehat{Z}} \ch_{\mx_\gl}^{2j} (\A_{\chi})
\end{eqnarray} 
where $N=\dim \cn$.  For the Springer sheaf, we have
\begin{eqnarray}\label{eqn.stalks2}
H^{2j+N}(\spr{\mx_\gl}) \cong \ch_{\mx_\gl}^{2j} ( \A_{\mathrm{id}}).
\end{eqnarray} 

Given a constructible sheaf $\ce$ on $\cn$ whose stalks have only even-dimensional cohomology, and $\mx\in \cn$,
we define the \textit{Poincar\'e polynomial of the stalk of $\ce$ at $\mx$} by
\[
P_{\mx}(\ce;t):= \sum_{j} \dim \ch_{\mx}^{2j} (\ce) t^j.
\]
Considering the $Z$-action on both sides of~\eqref{eqn.stalks}, and restricting to the $\chi$-isotypic component, we obtain the following consequence of Theorem~\ref{thm.poincare}.

\begin{Thm}\label{thm.lusztigsheaf} Let $\chi\in \widehat{Z}$ be the character of $Z$ defined by $\diag(a,a,\ldots, a) \mapsto a^i$,
 and let $d = n/\gcd(i,n)$. Then $P_{\mx_\lambda}(\A_\chi;t)=0$ unless $d$ divides $\lambda$, in which case
\[
P_{\mx_\lambda}(\A_\chi; t) = t^{N+D_{\lambda,d}} P( \spr{\mx_{\lambda/d}}; t).
\]
\end{Thm}
\begin{proof} We have $P_{\mx_\lambda}(\A_\chi; t) = t^N P_{\chi}(\espr{\gl};t)$ from the $Z$-equivariance of the decomposition~\eqref{eqn.stalks}.  The result now follows from Theorem~\ref{thm.poincare}.
\end{proof}

As a consequence, we obtain the following corollary, relating each Lusztig sheaf in type A to the Springer sheaf corresponding to a smaller rank group.  If $d|n$, let $N_d= N - N'$, where $N = \dim \cn$, and $N'$ is the dimension of the nilpotent cone for $SL_{n/d}(\C)$.
 
 \begin{Cor} \label{cor:smallergroup}
 Let $\chi\in \widehat{Z}$ be the character of $Z$ defined by $\diag(a,a,\ldots, a) \mapsto a^i$, and let $d = n/\gcd(i,n)$. Let $\A_{\mathrm{id}, n/d}$ denote the Springer sheaf of $SL_{n/d}(\C)$. For all partitions $\lambda$ such that $d$ divides $\lambda$, 
 \[
 P_{\mx_{\gl}}(\A_\chi;t) = t^{N_d + D_{\lambda,d}} P_{\mx_{\gl/d}}(\A_{\mathrm{id}, n/d};t).
 \] 
 \end{Cor}
 \begin{proof} We have
 $$
 P_{\mx_{\gl}}(\A_\chi;t) =  t^{N+D_{\lambda,d}} P( \spr{\mx_{\lambda/d}}; t) = t^{N-N'+D_{\lambda,d}} P_{\mx_{\gl/d}}(\A_{\mathrm{id}, n/d};t),
 $$
 where the second equality is from Theorem~\ref{thm.lusztigsheaf}, and the third from \eqref{eqn.stalks2} for $SL_{n/d}(\C)$.
 \end{proof}
 
 \begin{Rem}
 Given $\chi$, the Poincar\'e polynomial $ P_{\mx_{\gl}}(\A_\chi;t)$ can be computed using Lusztig's generalized
 Springer correspondence and the Lusztig--Shoji algorithm.  
 Here is an outline of the computation.
 The generalized Springer correspondence gives a decomposition
 $$
\A_\chi = \oplus IC(\co_{\lambda_i}, \cl_i) \otimes V_i,
 $$ 
 where $V_i$ is an irreducible representation of the relative Weyl group, whose dimension $r_i = \dim V_i$ can be
 computed.  We have
 $$
 P_{\mx_{\gl}}(\A_\chi;t) = \sum_i  r_i P_{\mx_{\gl}}(IC(\co_{\lambda_i}, \cl_i);t).
 $$
The Lusztig--Shoji algorithm gives a method to compute $P_{\mx_{\gl}}(IC(\co_{\lambda_i}, \cl_i);t)$,
 so the $P_{\mx_{\gl}}(\A_\chi;t) $ can be computed.   Using this method, the equality of Corollary \ref{cor:smallergroup} can
 be observed in examples.  It is possible that a proof of this corollary can be obtained by carefully tracing through the Lusztig-Shoji computation.  (See \cite{GM99} and \cite{Geck11} for information on how to compute examples.)
  \end{Rem}


\begin{thebibliography}{10}

\bibitem[Ach21]{Achar-book}
Pramod~N. Achar.
\newblock {\em Perverse Sheaves and Applications to Representation Theory},
  volume 258 of {\em Mathematical Surveys and Monographs}.
\newblock American Mathematical Society, 2021.

\bibitem[AHJR17]{AHJR1}
Pramod Achar, Anthony Henderson, Daniel Juteau, and Simon Riche.
\newblock Modular generalized springer correspondence i: the general linear
  group.
\newblock {\em J. Eur. Math. Soc.}, 18:1013--1070, 2017.

\bibitem[AM15]{Abe-Matsumura2015}
Hiraku Abe and Tomoo Matsumura.
\newblock Equivariant cohomology of weighted {G}rassmannians and weighted
  {S}chubert classes.
\newblock {\em Int. Math. Res. Not. IMRN}, (9):2499--2524, 2015.

\bibitem[Ful93]{Fulton}
W.~Fulton.
\newblock {\em {Introduction to Toric Varieties}}.
\newblock Princeton University Press, Princeton, New Jersey, 1993.

\bibitem[Gec11]{Geck11}
Meinolf Geck.
\newblock Some applications of chevie to the theory of algebraic groups.
\newblock {\em Carpathian Journal of Mathematics}, 27(1):64--94, 2011.

\bibitem[GM99]{GM99}
Meinolf Geck and Gunter Malle.
\newblock {On special pieces in the unipotent variety}.
\newblock {\em Experimental Mathematics}, 8(3):281 -- 290, 1999.

\bibitem[GPR23]{GPR}
William Graham, Martha Precup, and Amber Russell.
\newblock A new approach to the generalized {S}pringer correspondence.
\newblock {\em Trans. Amer. Math. Soc.}, 376(6):3891--3918, 2023.

\bibitem[Gra22]{Graham2019}
William Graham.
\newblock Toric varieties and a generalization of the {S}pringer resolution.
\newblock In {\em Facets of algebraic geometry. {V}ol. {I}}, volume 472 of {\em
  London Math. Soc. Lecture Note Ser.}, pages 333--370. Cambridge Univ. Press,
  Cambridge, 2022.

\bibitem[Jan04]{Jantzen}
Jens~Carsten Jantzen.
\newblock {\em Nilpotent Orbits in Representation Theory}, pages 1--211.
\newblock Birkh{\"a}user Boston, Boston, MA, 2004.

\bibitem[JMW14]{JMW14}
Daniel Juteau, Carl Mautner, and Georgie Williamson.
\newblock Parity sheaves.
\newblock {\em Journal of the American Mathematical Society}, 27(4):1169--1212,
  2014.

\bibitem[JP22]{Ji-Precup2022}
Caleb Ji and Martha Precup.
\newblock Hessenberg varieties associated to ad-nilpotent ideals.
\newblock {\em Comm. Algebra}, 50(4):1728--1749, 2022.

\bibitem[Lus86]{Lu86}
George Lusztig.
\newblock Character sheaves, {V}.
\newblock {\em Advances in Mathematics}, 61(2):103--155, 1986.

\bibitem[PT19]{Precup-Tymoczko}
Martha Precup and Julianna Tymoczko.
\newblock Springer fibers and {S}chubert points.
\newblock {\em European J. Combin.}, 76:10--26, 2019.

\bibitem[Sho87]{Sho87}
T.~Shoji.
\newblock Green functions on reductive groups over a finite field.
\newblock {\em The Arcata Conference on Representations of Finite Groups; Proc.
  Sympos. Pure Math.}, Part 1(47):289--302, 1987.

\bibitem[Spa77]{Spaltenstein1977}
N.~Spaltenstein.
\newblock {On the fixed point set of a unipotent element on the variety of
  Borel subgroups}.
\newblock {\em Topology}, 16(2):203 -- 204, 1977.

\bibitem[Spa82]{Spaltenstein}
N.~Spaltenstein.
\newblock {\em {Classes unipotentes et sous-groupes de Borel}}.
\newblock Springer-Verlag, 1982.

\bibitem[Spr76]{Springer1976}
T.A. Springer.
\newblock {Trigonometric sums, Green functions of finite groups and
  representations of Weyl groups}.
\newblock {\em Inventiones Mathematicae}, 36(1):173--207, 1976.

\bibitem[Tym06]{Tymoczko2006}
Julianna~S. Tymoczko.
\newblock Linear conditions imposed on flag varieties.
\newblock {\em Amer. J. Math.}, 128(6):1587--1604, 2006.

\end{thebibliography}


\ifx\undefined\bysame
\newcommand{\bysame}{\leavevmode\hbox to3em{\hrulefill}\,}
\fi

\end{document}